\documentclass[article,11pt]{amsart}
\usepackage[utf8]{inputenc}
\usepackage{amsfonts,amssymb}

\usepackage{graphicx}
\usepackage{tikz}
\usepackage{circuitikz}
\usepackage{euscript, psfrag,bm}
\usepackage{enumitem}
\usepackage{etoolbox}
\usepackage{amsmath}
\usepackage{mathtools}
\usepackage{enumerate}
\usepackage{amssymb}
\usepackage{amscd}
\usepackage{amsthm}
\usepackage{amsfonts}
\usepackage{graphicx}
\usepackage{tensor}
\usepackage{multirow}
\usepackage{multicol}
\usepackage{mathscinet}
\usepackage[T1]{fontenc}
\usepackage{mathrsfs}
\usepackage{fancyhdr}
\usepackage[colorlinks=true]{hyperref}
\hypersetup{urlcolor=blue, citecolor=blue}
\usepackage{cleveref}
\crefformat{section}{\S#2#1#3} 
\crefformat{subsection}{\S#2#1#3}
\crefformat{subsubsection}{\S#2#1#3}

\newtheorem{theorem}{Theorem}[section]
\newtheorem{corollary}[theorem]{Corollary}

\newtheorem{lemma}[theorem]{Lemma}

\newtheorem{proposition}[theorem]{Proposition}

\theoremstyle{definition}
\newtheorem{definition}[theorem]{Definition}

\newtheorem{notation}{Notation}

\newcommand{\Q}{\mathbb Q}
\newcommand{\R}{\mathbb R}
\newcommand{\Z}{\mathbb Z}

\DeclareMathOperator\diam{diam}

\DeclareMathOperator\dist{dist}

\DeclareMathOperator{\Hdim}{HDim}

\DeclareMathOperator{\hor}{hor}
\DeclareMathOperator{\NE}{NE}

\DeclareMathOperator{\pr}{pr}

\newcommand{\ds}{\displaystyle}
\newcommand{\eps}{\varepsilon}

\renewcommand{\Re}{\textrm{Re~}}
\renewcommand{\Im}{\textrm{Im~}}
\newcommand{\impl}{\quad\Longrightarrow\quad}

\title{Dichotomy for the Hausdorff dimension of nonergodic directions on translation surfaces}
\author{Yuming Wei}

\begin{document}

\begin{abstract}
We study the ergodic properties of the translation surface $X_{\lambda,\mu}$ formed by gluing two flat tori along a slit with holonomy \((\lambda,\mu) \in \mathbb{R}^2\). Extending the dichotomy result of Cheung, Hubert, and Masur for the case $\mu = 0$, we prove the following: for slits not parallel to any absolute homology class, the Hausdorff dimension of the set $\NE(X_{\lambda,\mu},\omega)$ of non-ergodic directions is either $0$ or $\frac{1}{2}$. This dichotomy is completely characterized by the P\'erez-Marco condition expressed in terms of best approximation denominators. As a corollary, we obtain that the P\'erez-Marco condition for best approximation denominators is norm-independent.
\end{abstract}
\maketitle
\section{Introduction}
Let \((X, \omega)\) be a translation surface, where \(X\) is a connected, compact Riemann surface of genus \(g\), and \(\omega\) is a non-zero holomorphic 1-form on \(X\). 
Consider the quadratic differential  $q_{\theta}=e^{2i\theta}\omega^2$ associated with $\theta$. The vertical foliation $\mathcal{F}_{\theta}$ of $q_\theta$ consists of leaves satisfying $\Re (\sqrt{q_{\theta}})=0$. To endow the vertical foliation $\mathcal{F}_{\theta}$ with a measure, we define the transverse measure $\mu_\theta$ as follows: for any transverse curve $\gamma$, the transverse measure is given by
\[
\mu_\theta(\gamma)=\int_{\gamma}|\Im(\sqrt{q_\theta})|=\int_{\gamma}|\Im(e^{i\theta}\omega)|.
\]
This makes $\mathcal{F}_\theta$ a measured foliation.  An important topological property of foliations is minimality. A foliation is called minimal if every closed union of leaves is either empty or the entire surface. The measure-theoretic counterpart of minimality is ergodicity.
A measured foliation is said to be ergodic if every measurable union of leaves has either zero measure or full measure. 
A stronger property than ergodicity is unique ergodicity. A measured foliation is called uniquely ergodic if it admits a unique (up to scaling) transverse invariant measure. 

The set of angles $\theta$ for which the measured foliation $\mathcal{F}_\theta$ is non-ergodic exhibits an intricate structure. We denote it as $\NE(X,\omega)$.
A fundamental result regarding the ergodic properties of these flows was established by Kerckhoff, Masur, and Smillie \cite{KMS}, which states that for any translation surface, the flow in almost every direction (with respect to the Lebesgue measure) is uniquely ergodic. While the set $\NE(X,\omega)$ is metrically rare (constituting a null set), its geometric organization displays potentially rich fractal behavior. In 1992, Masur \cite{Ma2} established an upper bound of \(\frac{1}{2}\) for the Hausdorff dimension of \(\NE(X,\omega)\).

Consider a special class of translation surfaces $(X_{\lambda,\mu},\omega)$ defined as the double cover of a standard flat torus of area one gluing along a slit with holonomy \((\lambda,\mu) \in \mathbb{R}^2\). 
A significant advancement occurred in 2003, when Cheung \cite{Ch1} proved that the upper bound Hausdorff dimension $\frac{1}{2}$ can be achieved in this case.
Later, in 2011, Cheung, Hubert, and Masur \cite{2011Dichotomy} demonstrated that \(\Hdim \NE(X_{\lambda,0})\) can only take the values \(0\) or \(\frac{1}{2}\). 
 
When the slit not parallel to any absolute homology class, a result of Veech \cite{Ve2} establishes that every minimal direction is uniquely ergodic if \((\lambda, \mu)\) is rational. Boshernitzan (See Appendix in \cite{Ch1}) showed that the existence of irrational vectors $(\lambda, \mu)$ for which $\NE(X_{\lambda,\mu},\omega)$ constitutes an uncountable but Hausdorff dimension $0$ set. In 2025, Huang \cite{Huang} provided a proof that the Hausdorff dimension of $\NE(X_{\lambda,\mu})$ is zero when the non-Pèrez-Marco condition holds, formulated in terms of best approximation denominators.

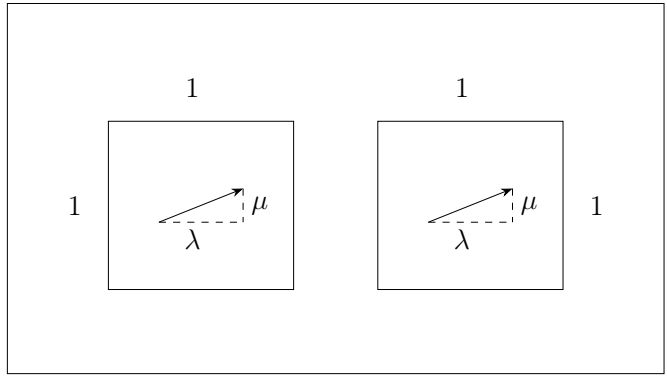
\begin{figure}[!ht]
\centering
\resizebox{0.7\textwidth}{!}{%
\begin{circuitikz}
\tikzstyle{every node}=[font=\large]
\draw[fill={rgb,255:red,255; green,255; blue,255}] (0,13.5) rectangle (9.75,8);

\draw  (1.5,11.75) rectangle (4.25,9.25);
\draw  (5.5,11.75) rectangle (8.25,9.25);
\draw [->, >=Stealth] (2.25,10.25) -- (3.5,10.75);
\draw [->, >=Stealth] (6.25,10.25) -- (7.5,10.75);
\draw [dashed] (2.25,10.25) -- (3.5,10.25);
\node [font=\large] at (2.75,10) {$\lambda$};
\draw [dashed] (3.5,10.75) -- (3.5,10.25);
\node [font=\large] at (3.75,10.5) {$\mu$};
\draw [dashed] (6.25,10.25) -- (7.5,10.25);
\node [font=\large] at (6.75,10) {$\lambda$};
\draw [dashed] (7.5,10.75) -- (7.5,10.25);
\node [font=\large] at (7.75,10.5) {$\mu$};
\node [font=\large] at (1,10.5) {1};
\node [font=\large] at (8.75,10.5) {1};
\node [font=\large] at (2.75,12.25) {1};
\node [font=\large] at (6.75,12.25) {1};
\end{circuitikz}
}%
\caption{$(X_{\lambda,\mu},\omega)$}
\label{fig:my_label}
\end{figure}

\subsection{Statement of main results}
In \cite{2011Dichotomy}, the dichotomy in Hausdorff dimension of nonergodic directions corresponds to whether the denominators of the continued fraction expansion satisfy the Pérez-Marco condition or not.
Our framework extends this result to the setting of best approximation vectors and establishes the following dichotomy:
\begin{theorem}\label{thm:dichotomy}
    Let \(\{q_k\}\) be the sequence of denominators of the best approximation vectors  of an irrational vector \((\lambda, \mu)\) with respect to any norm $\|\cdot\|$ on $\R^2$. Then \(\Hdim \NE(X_{\lambda,\mu},\omega) = 0\) or \(\frac{1}{2}\), with the latter case occurring if and only if 
    \begin{align}\label{PM:conv}
        \sum_k \frac{\log\log q_{k+1}}{q_k} < \infty.
    \end{align}
\end{theorem}

In Theorem~\ref{thm:dichotomy}, the dichotomy is completely characterized by the P\'erez-Marco condition (see \cite{PM}) expressed in terms of \textit{best approximation vectors}. For a given norm \( \|\cdot\| \) on \( \mathbb{R}^n \), we recall that

\begin{definition}\label{def:best approximation vector}
$\frac{\mathbf{p}}{q}\in \mathbb{Q}^n$ is a \textit{best approximation vector} of $\mathbf{x}$ if 
\begin{itemize}
    \item[(i)] $\|q\mathbf{x}-\mathbf{p}\|<\|q^\prime \mathbf{x}-\mathbf{p^\prime}\|$ for any $\left(\mathbf{p}^\prime,q^\prime\right)$ with $q^\prime<q$.
    \item[(ii)] $\|q\mathbf{x}-\mathbf{p}\|\leq\|q\mathbf{x}-\mathbf{p^\prime}\|$ for any $\mathbf{p^\prime}\in \mathbb{Z}^n$.
\end{itemize}
\end{definition}

Best approximation vectors have been introduced since a long time inside proofs in an nonexplicit form. In 1951, C. A. Rogers \cite{Rog} firstly  defined best Diophantine approximations with respect to the supreme norm.
Lagarias \cite{lag1} \cite{lag2} undertook a systematic study of the best Diophantine approximations. For more 
details, we recommend the survey given by
Moshchevitin (See \cite{Mos}). 

A long-recognized property,
in multidimensional simultaneous Diophantine approximation, is that the definition (of best approximations) depends intrinsically on the choice of norm. In the study of the Lévy–Khintchine constant arising from the sequence of denominators of best approximation vectors, Cheung and Chevallier \cite{Ch-Che} posed an intriguing question: do these constants coincide under different norms?  As a corollary of Theorem~\ref{thm:dichotomy}, we show that the sequence of denominators of best approximation vectors satisfies the P\`erez-Marco condition with  norm-independence:
\begin{corollary}\label{cor:independence of norm}
    All sequences of denominators of best approximation vectors of any irrational vector $(\lambda,\mu)$ satisfy the same \emph{P\'erez-Marco} condition.
\end{corollary}
\begin{proof}
    In the proof of Theorem~\ref{thm:dichotomy}, we showed that the Hausdorff dimension of $\NE(X_{\lambda,\mu},\omega)$ is entirely determined by the P\`erez-Marco condition satisfied by the sequence of denominators of best approximation vectors of $(\lambda,\mu)$, constructed with respect to any chosen norm $\|\cdot\|$ on \( \mathbb{R}^2 \).
\end{proof}
\subsection{Sketch of main ideas}
A key contribution of this work is the establishment of the Hausdorff dimension $\frac{1}{2}$, which had previously remained open.
To show the result, we define a recursive process to construct a tree of slits and associate this tree with a Hausdorff dimension of $\frac{1}{2}$ Cantor set. 
The construction of the tree of slits is divided into two parts based on the heights of the slits: the Liouville construction and the Diophantine construction. The Liouville construction is based directly on the best approximation vectors of $(\lambda, \mu)$.
The primary challenge arises from the Diophantine construction. To ensure a sufficiently large dimension while maintaining the self-sustaining nature of the construction, it is necessary to work with \emph{normal slits} at this stage. This paper establishes a sufficient condition for the normality of a slit 
\( w = (\lambda + m, \mu + n) \). The proof can be summarized as a strategy for identifying the best approximation vectors of \( (\lambda, \mu) \) under the assumption
:
\begin{align}\label{align:intro:p/q}
    \left| \frac{\lambda + m}{\mu + n} - \frac{p}{q} \right| < \frac{1}{q q'}
\end{align}
where \( q \) and \( q' \) are successive convergents of the inverse slope of \( w \), with \( |w| < q < |w|^{r} < |w|^{1+(r-1)N'} < q' \).  Our approach to the construction is based on the following idea: the best approximation vector we construct  arises from the intersection of two non-parallel rational affine lines, each of which lies at a very small distance from \( (\lambda, \mu) \).

The condition (\ref{align:intro:p/q}) yields the inequality
\[
|q\lambda - p\mu + mq - np| < \frac{\mu + n}{q'} < \frac{1}{|w|^{(r-1)N'}},
\]  
which geometrically corresponds to the line $\ell$ defined by  
\begin{align}\label{align:intro:ell}
    qx - py + mq - np = 0
\end{align}
with a small distance from $(\lambda,\mu)$.

The key novelty in our method lies in introducing the concept of \textit{Nearest affine ration lines} \footnote{This can also be referred to as the best linear forms of \( (\lambda, \mu) \)} to $(\lambda,\mu)$, through which we construct an additional rational line.
For the norm $\|\cdot\|$ we choose on $\mathbb{R}^n$,  we can define the height of $\eta =\left(\Vec{v},c\right)\in \mathbb{Z}_{\pr}^{n+1}$ \footnote{We use $\mathbb{Z}_{\pr}^{n+1}$ to denote the set of primitive vectors in $\mathbb{Z}^{n+1}$.} by $\eta^{-}=\|\Vec{v}\|$. 
\begin{definition}
For $\Vec{x}\in \mathbb{R}^n$, we say $\eta=\left(\Vec{v},c\right)\in \mathbb{Z}_{\pr}^{n+1}$ is a nearest rational affine hyperplane of $\Vec{x}$ if 
\begin{itemize}
  \item[(i)] $|\Vec{x}\cdot\Vec{v}+c|<|\Vec{x}\cdot\Vec{v^{\prime}}+c^\prime|$ for any $\eta^{\prime}=\left(\Vec{v^{\prime}},c^\prime \right)$ with height ${\eta^{\prime}}^-<\eta^-$.
  \item[(ii)]  $|\Vec{x}\cdot\Vec{v}+c|\leq |\Vec{x}\cdot\Vec{v^{\prime}}+c^\prime|$
for any  $\eta^{\prime}=\left(\Vec{v^{\prime}},c^\prime \right)$ with height ${\eta^{\prime}}^-=\eta^-$.
\end{itemize} 
\end{definition}
We fix once and for all a sequence $\{\eta_j\}_{j\in \mathbb{N}}\subseteq \Z_{\pr}^{n+1}$ of rational affine hyperplanes nearest to $\left(\lambda,\mu\right)$ with increasing heights $\eta_1^- < \eta_2^-< \cdots$. 
For irrational vectors in $\R^2$, we have the following dichotomy: If \( (\lambda, \mu) \) lies on a uniquely rational line, we call it uniquely rational; otherwise, it is totally irrational. 

Accordingly, we handle each case individually in the proof.\\
Case 1: If \( (\lambda, \mu) \) is uniquely rational, then it lies on a uniquely rational line \( \eta \), which can be expressed as 
\begin{align}\label{align:intro:uni. rat.}
    ax + by + c = 0,
\end{align}
where \( a, b, c \in \mathbb{Z} \) with \( a \neq 0 \) or \( b \neq 0 \).
\footnote{In \cite{2011Dichotomy}, equation (\ref{align:intro:uni. rat.}) reduces to \( y = 0 \).}
The best approximation vector we are looking for arises as the intersection of $\ell$ and $\eta$.
\\
Case 2:
When \( (\lambda, \mu) \) is totally irrational, the pair of nearest rational affine lines \( \eta_j \) and \( \eta_{j+1} \), whose heights satisfy \( \eta_j^- \leq \ell^- < \eta_{j+1}^- \), contribute to the construction of the best approximation vector.

\subsection{Outline of proof}
We provide a brief outline of the proof of Theorem~\ref{thm:dichotomy}, which naturally divides into two parts. The first part establishes an upper bound, demonstrating that the Hausdorff dimension is \( 0 \).
\subsubsection{Dimension 0 case}
A similar argument was given in \cite{Huang}, where the best approximation vectors of $(\lambda, \mu)$ were chosen with respect to the $\|\cdot\|_{\infty}$ norm. Here, we prefer to rewrite the proof to demonstrate that the Hausdorff dimension zero result holds for any  norm $\|\cdot\|$.

Under the assumption 
\begin{align}\label{PM:div}
      \sum_k\frac{\log\log q_{k+1}}{q_k}=\infty,
\end{align}
we start the proof by recalling the notions of {\em $Z$-expansions} and {\em Liouville direction} (relative to $Z$).
Then we associate every nonergodic direction $\theta\in\NE(X_{\lambda,\mu})$ with its $Z$-expansion, which is a sequence of slits $\{w_j\}$ and loops $\{v_j\}$ whose slopes converge to $\theta$ and satisfy the summability condition (\ref{ieq:sumx}). (See Theorem~\ref{thm:CE}.) 
Finally we show every $Z$-expansion is Liouville relative to $Z$ and then by Corollary~\ref{cor:LVdir}, the set of Liouville directions has Hausdorff dimension zero, we finish the proof.
This is stated as the Lemma~\ref{lem:NE->LV}.

The proof strategy for Lemma \ref{lem:NE->LV} is as follows: assuming $\theta$ is Diophantine relative to $Z$, the number of slits in a long interval $[q_k, q_{k+1})$ that satisfy the conditions of Lemma \ref{Lemma:Liouville convergent} is approximately $\log \log q_{k+1}$. By Lemma \ref{lem:min:area}, the sum of the vector cross products $|w_j \times v_j|$ over these slits is bounded below by $\frac{\log \log q_{k+1}}{q_k}$. Summing over all such long intervals indexed by $k$ and using the divergence condition (\ref{PM:div}), we obtain
\[
\sum \frac{\log \log q_{k+1}}{q_k} = \infty.
\]
This result contradicts the summability condition (\ref{ieq:sumx}).

\subsubsection{Dimension $1/2$ case}
The starting point for the argument establishing a dimension of \( \tfrac{1}{2} \) is Theorem~\ref{thm:sumx}, which states that if a sequence of slits \( \left\{w_j\right\} \) satisfies the summability condition~(\ref{ieq:sumx}), then the sequence of their inverse slopes converges to a nonergodic direction. 

Our strategy involves constructing a tree structure composed of slits, where each slit \( w_{j+1} \) in the \((j+1)\)-th level is derived from a slit \( w_j \) in the \( j \)-th level. These slits satisfy the conditions \( |w_{j+1}| \approx |w_j|^r \) and \( |w_j \times w_{j+1}| < \delta_j \), where \( \delta_j \) is chosen such that \( \sum \delta_j < \infty \). The slits \( w_{j+1} \) generated from \( w_j \) under these constraints are referred to as the \emph{children slits} of \( w_j \), while \( w_j \) is referred to as the \emph{parent slit}.

The first construction in this paper is designed to address the case of a slit \( w \) whose height falls within a short interval \([q_k, q_{k+1}]\) where $q_{k+1}<q_k^N$ for some fixed $N$.  
Our objective is to demonstrate that for an \(\alpha\)-normal slit \(w\), the child slits \(w'\) that are \(\alpha r\)-normal constitute a substantial portion (see Proposition~\ref{prop:normal}).  
First, we rule out the possibility of \(w'\) being a ``miracle slit'' by showing that \(w\) can have at most one miracle child slit (see Lemma~\ref{Lem:danger slit}).  
The proof then proceeds in two parts. In the first part, if the spectrum of \(w\) lies within the interval \((\alpha \rho^{N^\prime+1}|w|, |w|^r)\), we establish that \(w'\) is \(\alpha r\)-normal using Lemma~\ref{lem:good:children}, Lemma~\ref{lem:Uniquly N'-good=>normal} and  Lemma~\ref{lem:N'-good=>normal}.  
In the second part, we derive an upper bound for the number of child slits that are not \(\alpha r\)-normal. This is achieved through Lemma~\ref{lem:same:cluster} and Lemma~\ref{lem:cluster:size}.  
This approach to constructing ``normal child slits'' is termed the Diophantine construction.

In \S\ref{s:Liouville}, we introduce the construction using Liouville convergents of $w$ when $|w|$ falls in a long interval of $[q_k,q_{k+1}]$. We call it Liouville construction.
Lemma~\ref{Lemma:Liouville convergent} tells the liouville convergent $u$ is a ``very well approximation'' of $w$ in this case which implies a very small cross-product $|w\times u|$. Based on it, we have a set of $v$ constructed by $u$ such that $w'=w+2v$ is a child of $w$.

Then we associate this tree with a cantor set $F(r)$ and demonstrate that the lower bound of \(\Hdim F(r)\) can be arbitrarily close to \(\tfrac{1}{2}\) by choosing the parameter \(r\) sufficiently close to one.  
The detailed construction of the tree of slits and the precise definitions of the terms \(\delta_j\) are provided in \S\ref{s:Tree}. Finally, in \S\ref{s:Lower}, we verify the convergence of the series \(\sum \delta_j\).

\section{Preliminary}\label{s:NEdir} 
\subsection{Slits and loops}
Suppose that \( z_0 \) and \( z_1 \) are two marked points on a standard area-one flat torus \( T \), with horizontal distance \( \lambda \) and vertical distance \( \mu \).
Then a \emph{saddle connection} on $T$ is a straight line that starts and ends at $\{z_0,z_1\}$ without meeting either point in its interior. The set of saddle connections can be divided into two types: a \emph{slit}, which is a saddle connection joining \(z_0\) and \(z_1\), and a \emph{loop}, which is a saddle connection that joins one of these points to itself.
 
The holonomies of saddle connections are represented as pairs of real numbers.  
The set of holonomies of loops is given by  
\[
V_0 = \{(p, q) \in \mathbb{Z}^2 : \gcd(p, q) = 1\},
\]
and the set of holonomies of slits is expressed as  
\[
V_1 = V_1^+ \cup \left(-V_1^+\right),
\]
where  
\[
V_1^+ = \{(\lambda + m, \mu + n) : m, n \in \mathbb{Z}, n > 0\} \cup \{(\lambda, \mu)\}.
\]

Let $(X_{\lambda,\mu},\omega)$ denote the double cover of $T$ branched over $z_0$ and $z_1$, then we say that a slit $\gamma$ is \emph{separating} if $\gamma$ separates $X_{\lambda,\mu}$ into a pair of tori interchanged by the involution of the double cover. A lemma from \cite{Ch1} associates a separating slit with a pair of even integers \(m\) and \(n\), and then the collection of separating slits is said to have holonomies given by
\[
V_2 = V_2^+ \cup \left( -V_2^+ \right),
\]
where
\[
V_2^+ = \{ (\lambda + 2m, \mu + 2n) : m, n \in \mathbb{Z}, n > 0 \} \cup \{ (\lambda, \mu) \}.
\]

We then use the following definition to establish a connection between slits and loops.
\begin{definition}
We say two slits $w$ and $w'$ are related by a Dehn twist about $v$ if $w'-w\in \Z v$ for $v$ is a loop in $V_0$.
\end{definition}
  
The next theorem provides a method to construct non-ergodic directions on \((X_{\lambda,\mu}, \omega)\). It is a key theorem in the proof of the case for dimension \(\frac{1}{2}\).
\begin{theorem}\cite{MS}\label{thm:sumx}\footnote{A more general condition developed in \cite{MS} that applies to arbitrary translation surfaces and quadratic differentials. }
Let $\{w_j\}$ be a sequence of separating slits with increasing 
  heights $|w_j|$ and suppose that every consecutive pair of slits 
  $w_j$ and $w_{j+1}$ are related by a Dehn twist about some $v_j$ 
  such that \footnote{The cross-product formula expresses the area of the 
  parallelogram spanned by $u$ and $v$ as 
  $$|u\times v|=\|u\| \|v\|\sin\theta$$ 
  where $\times$ denote the standard skew-symmetric bilinear form on $\mathbb{R}^2$, 
  $\|\cdot\|$ the Euclidean norm, and $\theta$ the angle between $u$ and $v$.  
}
\begin{equation}\label{ieq:sumx}
  \sum_j |w_j\times v_j| < \infty.  
\end{equation}
Then the inverse slopes of $w_j$ converge to some $\theta$ and 
  this limiting direction belongs to $\NE(X_{\lambda,\mu},\omega)$.  
\end{theorem}

\subsection{Best approximation vectors}
\textit{Continued fractions} have a long history and were traditionally used to generate optimal rational approximations of irrational numbers. A continued fraction
is an expression of the form:
\[
x = a_0 + \cfrac{1}{a_1 + \cfrac{1}{a_2 + \ddots}}
\]
where $a_0$ is an integer and $a_1, a_2, \ldots$ are positive integers called \textit{partial quotients}. The rational number
\[
\frac{p_k}{q_k} = [a_0; a_1, a_2, \ldots, a_k] = a_0 + \cfrac{1}{{\ddots + \cfrac{1}{a_k}}}.
\]
is called the \( k \)-th \textit{convergent}.
We recall the following two classical facts from the theory of continued fractions, which yield effective estimates for how well irrational numbers can be approximated by rationals.
\begin{theorem}\cite[Thm. 9 and 13]{Kh}\label{thm:continued1}
    if $\{\frac{p_k}{q_k}\}$ is a sequence of convergents of $\theta$, then we have
    \begin{align}\label{align:estimate for convergents}
        \frac{1}{q_k(q_k+q_{k+1})}<|\theta-\frac{p_k}{q_k}|\leq\frac{1}{q_kq_{k+1}}.
    \end{align}
\end{theorem}
\begin{theorem}\cite[Thm. 19]{Kh}\label{lem:estimate for continued fractions}
    If a reduced fraction satisfies
    \begin{align}\label{convergents}
        |\theta-\frac{p}{q}|<\frac{1}{2q^2}
    \end{align}
    then $\frac{p}{q}$ is a convergent of $\theta$.
\end{theorem}

In Definition~\ref{def:best approximation vector}, we introduce the notion of best approximation vector. We now describe two theorems of best approximation vectors, which respectively generalize Theorems~\ref{thm:continued1} and~\ref{lem:estimate for continued fractions}.
\begin{definition}   
Let \( v \) be a vector in
\[
Q := \left\{ (\mathbf{p}, q) \in \mathbb{Z}^3 : \gcd(\mathbf{p}, q) = 1,\ q > 0 \right\}.
\]
We then define \( \Tilde{Q} \) as the set
\[
\Tilde{Q} := \left\{ \dot{v} = \frac{\mathbf{p}}{q} : v \in Q \right\}.
\]
\end{definition}
The statements of these two theorems are independent of any particular choice of norm; they hold uniformly for all norms. This generality enables us to derive Corollary~\ref{cor:independence of norm} from Theorem~\ref{thm:dichotomy}.
Further details can be found in \cite{Ch3}.

Henceforth, we fix a norm on \( \mathbb{R}^2 \), which we denote by \( \| \cdot \| \).

\begin{theorem}\cite[Theorem 2.13]{Ch3}
    Let $\dot{v_j}=\frac{\mathbf{p_j}}{q_j}\in \Q^2$ be a sequence of best approximation vectors to $\mathbf{x}$. Then
\begin{align}\label{C_2}
    \frac{1}{2q_{j+1}}\leq \|q_j\mathbf{x}-\mathbf{p_j}\| \leq \frac{C_1}{q_{j+1}^{1/2}}.
\end{align}
Here, \( C_1 \) denotes a constant that depends only on the choice of norm.
\end{theorem}

\begin{theorem}\cite[Theorem 2.10]{Ch3}\label{thm:proof of best}
If $\dot{v}=(\frac{\textbf{p}}{q})\in\Tilde{Q}$ satisfies
\begin{align}
    \|x-\dot{v}\|<\frac{1}{2q^2},
\end{align}
then $\dot{v}$ is a best approximation vector of $x\in \R^2$.
\end{theorem}

\section{Hausdorff Dimension 0}
In this section, the central idea of our proof is to demonstrate that, for any norm, the sequence of best approximation vectors satisfying condition~(\ref{PM:div}) yields $\Hdim \NE(X_{\lambda,\mu},\omega) = 0$. 

\subsection{Z-expansion}\label{s:Z-expansion} 
We recall the notions of $Z$-expansions and Liouville direction relative to a closed discrete subset $Z\subset \R^2$. These notions were firstly introduced in \cite{2011Dichotomy}.  

\begin{definition}
Given an inverse slope $\theta$ and $v=(p,q)\in\R^2$ we define 
  $$\hor_\theta(v)=|q\theta-p|\ \text{and}\ |v|=q.$$ 
$q$ is called the height of $v$.
Let $Z$ be a closed discrete subset of $\R^2$. A $Z$-\emph{convergent} of $\theta$ is any vector $v\in Z$ that minimizes the expression $\hor_\theta(u)$ among all vectors $u\in Z$ with $|u|\le|v|$. The $Z$-\emph{expansion} of $\theta$ is defined to be the sequence of $Z$-convergents ordered by increasing height. 
\end{definition}
\begin{notation}
Recall that $|v|$ is the absolute value of the $y$-coordinate. We call it the \emph{height} of $v$. If two or more $Z$-convergents have the same height we choose one and ignore the others.  
\end{notation}
In this paper, we set \( Z = V_0 \cup V_2 \). For the general definition of \( Z \), certain assumptions must be considered. Further details on this can be found in \cite[\S 3]{2011Dichotomy}.

\begin{definition}
A direction is called \emph{minimal} (relative to $Z$) if it is not the direction of vectors in $Z$.  
\end{definition}

The next lemma provides a sufficient condition on its \( Z \)-expansion to guarantee a minimal direction.

\begin{lemma}\cite[Lemma 3.3]{2011Dichotomy}\label{lem:Minkowski}
The $Z$-expansion of a direction with inverse slope $\theta$ is infinite if and only if $\theta$ is minimal.  
\end{lemma}

The following characterisation of nonergodic directions in terms of $Z$-expansions is a one of the key ideas used to prove $\Hdim\NE(X_{\lambda,\mu},\omega)=0$.  

\begin{theorem}\label{thm:CE}
(\cite{CE}) 
Let $\theta$ be a minimal\footnote{This implies it will also be 
a minimal direction relative to $Z$.} direction in $P_{\lambda,\mu}$.  
Then $\theta$ is nonergodic if and only if its $Z$-expansion is eventually alternating between loops and separating slits 
  $$\ldots, v_{j-1}, w_j, v_j, w_{j+1}, \ldots$$ 
and satisfies the summable cross-products condition (\ref{ieq:sumx}).  
\end{theorem}

\subsection{Liouville directions}\label{ss:LVdir}
\begin{definition}
A minimal direction is \emph{Diophantine} relative to $Z$ 
if its $Z$-expansion satisfies 
\begin{equation}\label{def:D_N}
  |w_{k+1}|=O\left(|w_k|^N\right)
\end{equation}
for some $N$.  And if no such $N$ satisfy (\ref{def:D_N}), it is \emph{Liouville} relative to $Z$.  
\end{definition}

The next lemma states that the collection of Liouville directions is a $0$ Hausdorff dimension set. Then the remaining task is to show that every direction $\theta\in \NE(X_{\lambda,\mu},\omega)$ is Liouville relative to $Z$. 

\begin{lemma}\cite[Corollary 3.9]{2011Dichotomy}\label{cor:LVdir}
The set of Liouville directions relative to the set of holonomies of saddle connections on a translation surface has Hausdorff dimension zero.  
\end{lemma}

By Lemma~\ref{cor:LVdir}, it is enough to show that every nonergodic direction is Liouville relative to $Z$.  We argue by contradiction and suppose $\theta\in \NE(X_{\lambda,\mu},\omega)$ is Diophantine relative to $Z$. We shall exploit (\ref{PM:div}) together with the Diophantine property to get a contradiction to (\ref{ieq:sumx}).

\begin{subsection}{Liouville convergents}
From now on, we set $\Tilde{C}=16C_1$.
Next Lemma shows that the best approximation vectors of $(\lambda,\mu)$ with $q_{k+1}\gg q_k$ give rise to the convergents of $\frac{\lambda+m}{\mu+n}$.

\begin{lemma}\label{Lemma:Liouville convergent}
    Assume that $w=(\lambda+m,\mu+n)$ is a slit, and $(\frac{p_{k,1}}{q_k},\frac{p_{k,2}}{q_k})$ is a best approximation vector of $(\lambda,\mu)$. 
    Let $\frac{p}{q}$ denote the fraction $\frac{p_{k,1}+mq_k}{p_{k,2}+nq_k}$ in lowest terms. If the height of $w$ such that
    \begin{align}\label{height restriction}
        |w|<\frac{q_{k+1}^{1/2}}{\Tilde{C}q_k}.
    \end{align}
    then $\frac{p}{q}$ is a convergent of $\frac{\lambda+m}{\mu+n}$ and the height of next convergent is larger than $\frac{q_{k+1}^{1/2}}{2C_1}$.    
\end{lemma}
\begin{proof}
 Let $v=(p,q)$ and $\mathbf{x}=(\lambda,\mu)$. 
 Since
    \begin{align*}
        \frac{|w\times v|}{|w||v|}   
        &=|\frac{\lambda+m}{\mu+n}-\frac{p_{k,1}+mq_k}{p_{k,2}+nq_k}|\\
        &=\frac{|(\lambda+m)(p_{k,2}-q_k\mu)-(\mu+n)(p_{k,1}-q_k\lambda)|}{|w|(p_{k,2}+nq_k)},
    \end{align*}
    we have
    \begin{align}\label{align:upper bound for liouville convergent}
        \frac{|w\times v|}{|w|}\leq 4\|q_k\mathbf{x}-\mathbf{p_k}\|\leq \frac{4C_1}{q_{k+1}^{1/2}}
    \end{align}
    for $(\mathbf{p_k},q_k)=(p_{k,1},p_{k,2},q_k)$.
    While $|v|\leq nq_k+p_{k,2} \leq 2q_k|w|$,
    then 
    $$2|v|\frac{|w\times v|}{|w|}\leq \frac{\Tilde{C}q_k|w|}{q_{k+1}^{1/2}}.$$
    So if $|w|< \frac{q_{k+1}^{1/2}}{\Tilde{C} q_k}$, we will get
    $$\frac{|w\times v|}{|w||v|}<\frac{1}{2|v|^2}$$
    which implies that $\frac{p}{q}$ is a convergent of $\frac{\lambda+m}{\mu+n}$ by Theorem~\ref{lem:estimate for continued fractions}.
    
Let $q^\prime$ be the height of next convergent of $\frac{\lambda+m}{\mu+n}$. By (\ref{align:estimate for convergents}), we have
\begin{align*}
    q^\prime>\frac{2|w|}{|w\times v|}
\end{align*}
which implies $$q^\prime>\frac{q_{k+1}^{1/2}}{2C_1}.$$
\end{proof}

\begin{definition}\label{def:Liouville convergent}
    When the conclusion of Lemma \ref{Lemma:Liouville convergent} holds, we refer to $\frac{p}{q}$ (or the vector $v=(p,q)$) as the \textit{Liouville convergent} of $w$ indexed by k.
\end{definition}

We denote $$n_k=\frac{1}{2}\log_{q_k}{q_{k+1}}-2$$ for each $k$. Then $\left\{n_k\right\}$ is a divergent sequence. If not, for some N, we have 
$$\sum \frac{\log\log q_{k+1}}{q_k} \le \sum \frac{\log 2N+\log\log q_k}{q_k}< \infty.$$
Let 
\begin{align}
    S_k= \left\{w, q_k<|w|<q_k^{n_k}\right\},
\end{align}
the next lemma gives a sufficient condition on the non-zero upper bound for the cross product $|w \times v|$ where $w \in V_2$ and $v\in V_0$.

\begin{lemma}\label{lem:min:area}
Let $\theta$ be a minimal nonergodic  direction and $w_j$, $w_{j+1}$, $w_{j+2}$ be three consecutive separating slits in its $Z$-expansion.
Suppose further that $w_j,w_{j+1},w_{j+2}\in S_k$ for some $k$, then we can deduce that
\begin{align}\label{aligh:cross products of u and v}
    |w_j\times v_j|>\frac{1}{4q_k}.
\end{align}    
\end{lemma}

\begin{proof}
By applying Theorem~\ref{thm:CE}, we can write $w_{j+1}=w_j+bv_j$ for some nonzero, even integer b. Then
    \begin{align*}
        |v_j|=\frac{|w_{j+1}-w_j|}{|b|}\leq\frac{|w_{j+1}|+|w_j|}{2}<|w_{j+1}|.
    \end{align*}
    Let $\alpha^\prime$ be the inverse slope of $w_{j+1}$ and $v_j=(p,q)$, we have
    \begin{align*}
        |\alpha^\prime-\frac{p}{q}|=\frac{|w_{j+1} \times v_j|}{|w_{j+1}||v_j|}<\frac{|w_j\times v_j|}{|v_j|^2}<\frac{1}{2q^2}.
    \end{align*}
    Then Theorem~\ref{lem:estimate for continued fractions} shows that $\frac{p}{q}$ is a convergent of $\alpha^\prime$. Let $q^\prime$ be the height of next convergent of $\alpha^\prime$. Then (\ref{align:estimate for convergents}) implies
    \begin{align*}
        \frac{1}{2qq^\prime}<|\alpha^\prime-\frac{p}{q}|<\frac{1}{qq^\prime}
    \end{align*}
    equivalently,
    \begin{align*}
        \frac{|w_{j+1}|}{2|w_j\times v_j|}<q^\prime<\frac{|w_{j+1}|}{|w_j\times v_j|}.
    \end{align*}
   and let $u$ be the Liouville convergent of $w_{j+1}$ indexed by k. $u$ can not have its height $|u|<q$ because Lemma \ref{Lemma:Liouville convergent} implies the height of next convergent of $\alpha^{\prime}$ is greater than $\frac{q_{k+1}^{1/2}}{2C_1}>|w_{j+1}|>|v_j|=q$, contracting the fact that q is also the height of a convergent of $\alpha^\prime$. Thus $|u|\geq |v_j|$.
    
    If we have $|u|>|v_j|$ so that $|u|\geq q^\prime>\frac{|w_{j+1}|}{2|w_j\times v_j|}$. Since $|u|\leq
    p_{k,2}+|w_{j+1}|q_k \leq 2|w_{j+1}|q_k$, the inequality (\ref{aligh:cross products of u and v}) is fulfilled.
    
    If not, we have $|u|=|v_j|$. Then by Lemma~\ref{Lemma:Liouville convergent}, we have
    $q^\prime>\frac{q_{k+1}^{1/2}}{C_1}$. By the definition of $v_{j+1}$, we have
    \begin{align*}
        1\leq |u\times v_{j+1}|&\leq|u||v_{j+1}|(\sin(\angle{w_{j+1} u})+\frac{|w_{j+1}\times v_{j+1}|}{|w_{j+1}||v_{j+1}|})\\
    &\leq \frac{|v_{j+1}|}{q^\prime}+\frac{|v_j|}{2|w_{j+1}|}<\frac{|v_{j+1}|}{q^\prime}+\frac{1}{2}
    \end{align*}
    from which it follows that $|v_{j+1}|>\frac{q^\prime}{2}>\frac{q_{k+1}^{1/2}}{\Tilde{C}}$. Then we can deduce that
    $$|w_{j+2}|>|v_{j+1}|>\frac{q_{k+1}^{1/2}}{\Tilde{C}}$$
    which contradicts with the height condition $|w_{j+2}|<q_k^{n_k}$.
\end{proof}
After preparing all ingredients, we begin to prove the main theorem of this section.
\begin{theorem}\label{lem:NE->LV}
    Assume 
    \begin{align*}
        \sum_{k}^{\infty}\frac{\log\log q_{k+1}}{q_k}=\infty
    \end{align*}
    holds. Then any minimal $\theta\in\NE(X_{\lambda,\mu},\omega)$ is Liouville relative to $Z$ which deduce that $\Hdim\NE(P_{\lambda,\mu}$) = 0.
\end{theorem}

\begin{proof}
We prove by contradiction and suppose $\theta$ is Diophantine relative to $Z$.  
Then there exists 
$N'>1$ for which the growth estimate $|w_{k+1}|<|w_k|^{N'}$ holds for all $k$. Then we have
$$\# \left\{w_j,  q_k<|w_j|<q_k^{n_k}\right\} \geq \left \lfloor\frac{\log (n_k)}{\log{N'}}\right\rfloor >\frac{\log n_k}{2\log N'}$$ 
provided $n_k > N_0$ for some $N_0$ depending only on $N'$. Applying Lemma~\ref{lem:min:area},
we thus obtain a contradiction since 
$$\sum_j|w_j\times v_j|\geq\sum_{n_k > N_0}\sum_{w_j\in S_k} |w_j\times v_j| \gtrsim \sum_{n_k>N_0} \frac{\log \log q_{k+1}}{q_k}=\infty.$$
\end{proof}
The last equation is given by (\ref{PM:div}) and $n_k$ is unbounded.
\end{subsection}

\section{Two approaches to the construction}\label{s: two construction}
We begin the proof of the Hausdorff dimension \(\frac{1}{2}\) result. Theorem~\ref{thm:sumx} establishes that the limiting set of \(\{w_j\}\), satisfying condition (\ref{ieq:sumx}), converges to a nonergodic direction on \((X_{\lambda,\mu}, \omega)\). Based on this result, our strategy is to construct a tree of slits satisfying (\ref{ieq:sumx}), associate this tree to a Cantor set, and then prove that the Cantor set has Hausdorff dimension \(\tfrac{1}{2}\).

This section presents two approaches to the construction of tree of slits. Specifically, we propose the following algorithm. 
We fix a large number $N$ from now on. Let \begin{align}\label{def:ell_N}
    \ell_N=\{k: q_{k+1}>q_k^N\}.
\end{align}
An interval $(q_k,q_{k+1})$ is called \textit{long}, if $k\in \ell_N$. Otherwise, we call it \textit{short}.
When a slit $w$ of height $|w|$ lies within  the long interval $[q_k, q_{k+1}]$, we will use the \textit{Liouville construction} to identify the children of $w$. We will introduce this method in detail in \S\ref{s:Liouville}. Let $k,k'$ be the consecutive  elements of $\ell_N$. 
The primary construction in this paper is designed to address the case of a slit \( w \) whose height falls within the union of short intervals \((q_{k+1}^{\frac{1}{N}}, q_{k'}^{\frac{1}{3r}})\). We refer  it as \textit{Diophantine construction}, is established by Proposition~\ref{prop:normal}.

In the following three sections, we assume that \(\ell_N\) is an infinite set. Under this assumption, both the Liouville construction and the Diophantine construction will be required. The case where \(\ell_N\) is a finite set will be considered in \S~\ref{s:A specail case}.
The choice of the parameter $N$ will depend on considerations in $\S \ref{ss:w_0}$ and will be specified there, by (\ref{def:N}). The indices in $\ell_N$ will guide our choice of constructions.

\subsection{Diophantine construction}\label{s:Diophantine}
let us first recall the definition of good slit.
Assume parameters $1<\alpha <\beta$ be given.  

\begin{definition}
A slit $w$ is $(\alpha,\beta)$-\emph{good} if its inverse slope has a convergent of height $q$ satisfying $\alpha|w|\le q\le \beta|w|$.  
\end{definition}

Let $\Delta(w,\alpha,\beta)$ be the collection of slits of the form $w+2v$ where $v\in\mathbb{Z}\times\mathbb{Z}_{>0}$ satisfies $\gcd(v)=1$ and 
\begin{equation}\label{def:Delta}
  \beta|w| \le |v| \le 2\beta|w| \quad\text{and}\quad 
  \frac{1}{\beta} < |w\times v| < \frac{1}{\alpha}.  
\end{equation}
The next lemma gives a lower bound for the number of such $w+2v$ constructed from good slits $w$. 

\begin{lemma}\cite[Lemma 7.2]{2011Dichotomy}\label{lem:good}
There is a universal constant $0<c_0<1$\footnote{To apply \cite[Thm.3]{Ch1} one needs to assume 
  $\beta\gg\alpha$, but this hypothesis was shown to be redundant in 
  \cite{Ch2}.  Indeed, by \cite[Thm.4]{Ch2} we can take $c'_0=\frac{4}{27\pi}.$}  such that 
\begin{equation}\label{number:good}
  \#\Delta(w,\alpha,\beta)\ge\frac{c_0\beta}{\alpha}.  
\end{equation}
for any $(\alpha,\beta)$-good slit $w$ and $\alpha<c_0\beta$.  
\end{lemma}
The next lemma shows that the good property will hold in a weaker vision for the child slit.
\begin{lemma}\label{lem:good:children}
If $w$ is an $(\alpha,\beta)$-good slit, then every 
slit $w'\in\Delta(w,\alpha,\beta)$ is $(\alpha-\frac12,\beta)$-good, but not $(1,\alpha-\frac12)$-good.  
\end{lemma}
\begin{proof}
Let $w'=w+2v\in\Delta(w,\alpha,\beta)$ where $w'=(\lambda+m',\mu+n')$ and $v=(p,q)$. Then we have 
$$\left|\frac{\lambda+m'}{\mu+n'}-\frac{p}{q}\right| = \frac{|w'\times v|}{|w'||v|}<\frac{|w\times v|}{2|v|^2}<\frac{1}{2\alpha q^2}<\frac{1}{2q^2}$$ 
for $\alpha>1$. Then by (\ref{convergents}), $v$ is a convergent of $w'$.  
Let $q'$ be the height of the next convergent of $w'$.  
Then by (\ref{align:estimate for convergents}), we have
\begin{equation}\label{ieq:cfc}
  \frac{1}{q'+q} < \frac{|w\times v|}{|w'|} < \frac{1}{q'}.
\end{equation}
From the left hand side above we have 
$$q' > \frac{|w'|}{|w\times v|} - q > (\alpha-\frac{1}{2})|w'|$$  
for $|w'|>2|v|=2q$, and from the right hand side of (\ref{ieq:cfc})
$$q' < \frac{|w'|}{|w\times v|} < \beta|w'|.$$  
This shows that $w'$ is $(\alpha-\frac12,\beta)$-good.  

Since $q$ and $q'$ are the heights of consecutive convergents of $w'$ (and since $|v|<|w'|$) 
it follows that $w'$ is not $(1,\alpha-\frac12)$-good.  
\end{proof}

We introduce a key definition characterizing a well-behaved Diophantine property. Let $r \in \mathbb{R}$ be a fixed constant satisfying the chain of inequalities:
\[
1 < r < r^3 < \frac{3}{2}.
\]
Based on this constant, we define two derived quantities:
\begin{align}\label{aligh: how deep for a given convergent} 
    N^{\prime} = \frac{4(N+1)r}{r-1},\ \ 
     \rho = r + \frac{1}{2}
\end{align}
where $N$ is a positive integer parameter.

\begin{definition}
A slit $w$ is $\alpha$-normal if it is $(\alpha\rho^t,|w|^{(r-1)t})$-good for all $t\in[1,T]$ 
where $T>1$ is determined by $\alpha \rho^T=|w|^{r-1}$.  
Equivalently, $w$ is $\alpha$-normal if and only if 
for all $t\in[1,T]$ we have 
\begin{equation}\label{def:normal}
  \Psi(w)\cap[\alpha \rho^t|w|,|w|^{1+(r-1)t}]\neq\emptyset.  
\end{equation}
where $\Psi(w)$ denotes the collection of heights of the 
convergents of the inverse slope of $w$.  
\end{definition}

We shall now present sufficient conditions for $w$ to be $\alpha$-normal. The detailed proof will be divided into two cases.

\begin{notation}
We note that there exist a global constant $C^\prime$ such that $\|\cdot\|\asymp_{C^\prime}\|\cdot\|_\infty$ where $||\cdot\|_\infty$ is the supremum norm. 
\end{notation}

\subsubsection{Uniquely rational case}
We first consider the case when $(\lambda, \mu)$ is uniquely rational, meaning that it lies on a uniquely rational line $\eta$, which can be expressed as
\begin{align}\label{def:eta}
ax + by + c = 0,
\end{align}
where $(a, b, c) \in \Z_{\pr}^3$ and at least one of $a$ or $b$ is nonzero.
Let $\eta^- \coloneqq \max\{|a|, |b|\}$.  
Hereafter, we always assume the parameter $k$ is chosen sufficiently large to satisfy $q_k > \eta^-$ \footnote{The selection of such a $k$ is permissible since $\ell_N$ constitutes an infinite set.}.
This leads to the following sufficient condition for $\alpha$-normality of a slit $w$:

\begin{lemma}\label{lem:Uniquly N'-good=>normal}
Let $w$ be a slit such that
\begin{align}\label{Uniquely Diophtine region}
       q_{k+1}^{\frac{1}{N}}\leq |w|<q_{k^\prime}^{\frac{1}{3r}}.
\end{align}
where $k,k'$ are consecutive elements of $\ell_N$.  If $w$ is $\left(\alpha\rho^{N^\prime},|w|^{r-1}\right)$-good, then it is $\alpha$-normal.
\end{lemma}
\begin{proof}
If $w$ is $\left(\alpha\rho^{N^\prime},|w|^{r-1}\right)$-good but not $\alpha$-normal. Let $v=(p,q)$ be the convergent of $w$ with the maximal height $q<|w|^r$. Since $w$ is $\left(\alpha\rho^{N^\prime},|w|^{r-1}\right)$-good, we have $\alpha\rho^{N^\prime}|w|\leq q<|w|^r$. Let $v^\prime=(p^\prime,q^\prime)$ be the next convergent. If $q^\prime\leq |w|^{1+(r-1)N^\prime}$, then $w$ is $\alpha$-normal by definition. We proceed by contradiction. Suppose that
\begin{align*}
q' > |w|^{1 + (r - 1)N'},
\end{align*}
Then 
\begin{align*}
 \frac{q^\prime}{|w|}>|w|^{(r-1)N^\prime}\geq q^{(1-\frac{1}{r})N^\prime}=q^{4(N+1)}.
\end{align*}
So it can be deduced that
\begin{align*}
  |w\times v|<\frac{|w|}{|v^\prime|}<\frac{1}{q^{4(N+1)}}.
\end{align*}
Writing $w=\left(\lambda+m,\mu+n\right)$ we have
\begin{align*}
    |q\lambda-p\mu+(qm-pn)|<\frac{1}{q^{4(N+1)}}.
\end{align*}
Let $\ell$ denote the line defined by
$$qx - py + (qm - pn) = 0.$$
We now consider the intersection of $\ell$ with $\eta$, as defined in \eqref{def:eta}. This intersection yields:
\begin{align*}
    \left(\frac{cp+(qm-pn)b}{ap+bq},\frac{cq+(pn-qm)a}{ap+bq}\right) =: \left(\frac{p_{h,1}}{q_h},\frac{p_{h,2}}{q_h}\right) =: \dot{v}_h
\end{align*}
in lowest terms.  
We also have $$q_h\leq |ap+bq| \leq 2\eta^-q\leq 2q^2$$
where the last inequality follows from $q > q_{k+1}^{1/N} > q_k > \eta^-$.
Then it follows that 
\begin{align*}
      \|(\lambda,\mu)-\dot{v}_h\|_{\infty}<\frac{2\max(|a+q|,|b+p|)}{q^{4(N+1)}|ap+bq|}<\frac{4}{q^{4N+3}}.
\end{align*}
By Theorem~\ref{thm:proof of best} and the definition of $C^\prime$ \footnote{Here, we use the definition of $q_{k_0}$ from \eqref{def:k_0} to eliminate the constant $C'$. We use this fact throughout the entire proof.},
    \begin{align*}
         \|(\lambda,\mu)-\dot{v}_h\| < C^\prime\|(\lambda,\mu)-\dot{v}_h\|_{\infty}<\frac{1}{2(2q^2)^2}<\frac{1}{2q_h^2}
    \end{align*}
     which implies that $\dot{v}_h$ is a best approximation vector.
     By (\ref{C_2}), we have
     \begin{align}
         \frac{1}{2q_hq_{h+1}}<\frac{4C'}{q^{4N+3}}.
     \end{align}
     Since $q>|w|>48C'$, we have
     \begin{align*}
         q_{h+1}>\frac{q^{4N+3}}{8C'q_h}>\frac{q^{4N+1}}{16C'}>q^{3N}>(4q^2)^N\geq(q_h)^N
     \end{align*}
     from which it follows that $h\in l_N$. Since $q_h\leq 2q^2\leq 2|w|^{2r}<|w|^{3r}<q_{k^\prime}$, we must have $q_h\leq q_k$. Hence $q_{h+1}<q_{k+1}$ so that
     \begin{align*}
         \frac{1}{2q_kq_{k+1}}\leq\frac{1}{2q_hq_{h+1}}<\frac{4C'}{q^{4N+3}}.
     \end{align*}
     Since $\alpha>1$, we have $q>|w|\geq q_{k+1}^{\frac{1}{N}}>q_k$. Then
     \begin{align*}
         q_{k+1}>\frac{q^{4N+3}}{12C'q_k}>q^{N}>|w|^{N}
     \end{align*}
     which contradicts with the hypothesis on $|w|$.
\end{proof}

\subsubsection{Totally irrational case}
The case when $(\lambda, \mu)$ is totally irrational presents the main difficulty in solving this problem. To address it, we use the notion of the nearest rational affine line to $(\lambda, \mu)$. We also need the definition of \textit{miracle slits} to exclude certain exceptional cases.
\begin{definition}
Given a slit $w=(\lambda+m,\mu+n)$ we associate to it the triple $\ell=(q,-p,qm-pn)\in\mathbb{Z}^3_{pr}$ where $p/q$ is the convergent of the inverse slope $\frac{\lambda+m}{\mu+n}$ determined by $|w| < q\le |w|^r < q'$ and $q'$ is the height of the next convergent.
Then we let $\eta_j$ be the nearest rational affine line to $(\lambda,\mu)$ determined by $\eta_j^- \le \ell^- < \eta_{j+1}^-$.  We refer to $(\ell,\eta_j)$ as the ``pair of lines'' associated to $w$.  
We say $w$ is a \textbf{miracle slit} if $\ell=\pm\eta_j$ and $\eta_{j+1}^- > (\eta_j^-)^{r^3}$. 

We will show in next Lemma that the miracle slits appear with a very low frequency in $\Delta(w,\alpha,\beta)$.
\end{definition}
\begin{lemma}\label{Lem:danger slit}
There exist some constant $c$ (depending only on norm) such that at most one miracle slit appear in  $\Delta(w,\alpha,\beta)$ if 
\begin{align}
        |w|^{(r-1)^2}>c.
\end{align}
\end{lemma}
\begin{proof}
let $w^\prime\in \Delta(w,\alpha,\beta)$ be a miracle slit with smallest norm and $(\ell,\eta_j)$ the pair of lines of $w^\prime$ such that  
$$\eta_{j+1}^->(\eta_j^-)^{r^3}=(\ell^-)^{r^3}.$$
Writing $\ell$ by $\ell=(q,-p,qm-pn)$, we have
$${q}^{1/r}\leq |w^\prime|<q$$
implies
$${\big(\frac{\eta_j^-}{C^\prime}\big)}^{1/r}\leq |w^\prime|< C^\prime \eta_j^-.$$
Then the next miracle slit $w^{\prime\prime}$ associating with $(\ell',\eta_{j'}$ ordered by ascending norm satisfies 
\begin{align}\label{danger slit}
        |w^{\prime\prime}|>{(\frac{\eta_{j'}^-}{C^\prime})}^{1/r}>\frac{(\eta_j^-)^{r^2}}{C^\prime}>\frac{|w^\prime|^{r^2}}{{C^\prime}^{5}}.
\end{align}    
Because $\beta|w|\leq |v| \leq 2\beta|w|$, each slit $\Tilde{w} \in \Delta(w,\alpha,\beta)$ satisfies $$(2\beta-1)|w|\leq |\Tilde{w}| \leq (4\beta+1)|w|.$$ 
We choose $c=5{C^\prime}^5$ and $|w|^{r^2-1}>|w|^{(r-1)^2}>c$, then     
$4\beta+1<5(2\beta-1)$ will deduce that $w^{\prime\prime}$ is not in  $\Delta(w,\alpha,\beta)$. 
\end{proof}

\begin{lemma}\label{lem:N'-good=>normal}
    Let $w$ be a slit which is not miracle such that\footnote{To prove this lemma, we require \( |w| \) to be sufficiently large so as to dominate certain constants (e.g., \( |w| > 48C' \)). This assumption is implicitly built into the choice of \( q_0 \) in (\ref{def:k_0}).}
    \begin{align}\label{Diophtine region}
        q_{k+1}^{\frac{1}{N}}\leq |w|<q_{k^\prime}^{\frac{1}{3r}}.
    \end{align}
   where $k,k'$ are consecutive elements of $\ell_N$.  If $w$ is $\left(\alpha\rho^{N^\prime},|w|^{r-1}\right)$-good, then it is $\alpha$-normal.
\end{lemma}
\begin{proof}
    We prove by contradiction. Assume $w$ is $\left(\alpha\rho^{N^\prime},|w|^{r-1}\right)$-good but not $\alpha$-normal. Let $v=(p,q)$ be the convergent of $w$ with the maximal height $q<|w|^r$. Since $w$ is $\left(\alpha\rho^{N^\prime},|w|^{r-1}\right)$-good, we have $\alpha\rho^{N^\prime}|w|\leq q<|w|^r$. Let $v^\prime=(p^\prime,q^\prime)$ be the next convergent. If $q^\prime\leq |w|^{1+(r-1)N^\prime}$, then $w$ is $\alpha$-normal by definition. So we must have
    \begin{align*}
        q^\prime>|w|^{1+(r-1)N^\prime}.
    \end{align*}
    Then 
    \begin{align*}
        \frac{q^\prime}{|w|}>|w|^{(r-1)N^\prime}\geq q^{(1-\frac{1}{r})N^\prime}=q^{4(N+1)}.
    \end{align*}
    So it can be deduced that
    \begin{align*}
        |w\times v|<\frac{|w|}{|v^\prime|}<\frac{1}{q^{4(N+1)}}.
    \end{align*}
    Writing $w=\left(\lambda+m,\mu+n\right)$ we have
    \begin{align*}
        |q(\lambda+m)-p(\mu+n)|<\frac{1}{q^{4(N+1)}}.
    \end{align*}
    Equivalently,
    \begin{align}\label{align:eta_l}
        |q\lambda-p\mu+(qm-pn)|<\frac{1}{q^{4(N+1)}}.
    \end{align}
Let $(\ell,\eta_j)$ be the pair of lines associated to $w$ for $l=(q,-p,qm-pn)$ by definition.  Since $w$ is not miracle, by hypothesis, we have only two cases to consider. We will show that each will lead to a contradiction.

Case 1: $\ell\neq\pm\eta_j$\\
 By the definition of nearest rational affine line, we have
    \begin{align}\label{aligh:eta_j}
        |a_j\lambda+b_j\mu+c_j|<\frac{1}{q^{4(N+1)}}.
    \end{align}
Note that $\ell$ and $\eta_j$ are not parallel. If not, we notice the distance between $\ell$ and $\eta$ is larger than $\frac{1}{2q}$. While the distances from $(\lambda,\mu)$ to $\ell$ and $\eta_j$ are both less than $\frac{1}{q^{4(N+1)}}$, which is a contradiction. 
       
The intersection point of $\ell$ and $\eta_j$ is 
\begin{align*}
    \left(\frac{c_jp+(qm-pn)b_j}{a_jp+b_jq},\frac{c_jq+(pn-qm)a_j}{a_jp+b_jq}\right) =: \left(\frac{p_{h,1}}{q_h},\frac{p_{h,2}}{q_h}\right) =: \dot{v}_h
\end{align*}
in lowest terms. Recalling that $\eta_j^-\le\ell^-$ 
we have $$q_h\leq |a_jp+b_jq| \leq 2\eta_j^-q\leq 4q^2.$$
From (\ref{align:eta_l}) and (\ref{aligh:eta_j}) it follows that 
    \begin{align*}
      \|(\lambda,\mu)-\dot{v}_h\|_{\infty}<\frac{2\max(|a_j+q|,|b_j+p|)}{q^{4(N+1)}|a_jp+b_jq|}<\frac{6}{q^{4N+3}}.
    \end{align*}
    By Theorem~\ref{thm:proof of best} and the definition of $C^\prime$,
    \begin{align*}
         \|(\lambda,\mu)-\dot{v}_h\| < C^\prime\|(\lambda,\mu)-\dot{v}_h\|_{\infty}<\frac{1}{2(4q^2)^2}<\frac{1}{2q_h^2}
    \end{align*}
     which implies that $\dot{v}_h$ is a best approximation vector.
     By (\ref{C_2}), we have
     \begin{align}
         \frac{1}{2q_hq_{h+1}}<\frac{6C'}{q^{4N+3}}.
     \end{align}
     Since $q>|w|>48C'$, we have
     \begin{align*}
         q_{h+1}>\frac{q^{4N+3}}{12C'q_h}>\frac{q^{4N+1}}{48C'}>q^{3N}>(4q^2)^N\geq(q_h)^N
     \end{align*}
     from which it follows that $h\in l_N$. Since $q_h\leq 4q^2\leq 4|w|^{2r}<|w|^{3r}<q_{k^\prime}$, we must have $q_h\leq q_k$. Hence $q_{h+1}<q_{k+1}$ so that
     \begin{align*}
         \frac{1}{2q_kq_{k+1}}\leq\frac{1}{2q_hq_{h+1}}<\frac{6C'}{q^{4N+3}}.
     \end{align*}
     Since $\alpha>1$, we have $q>|w|\geq q_{k+1}^{\frac{1}{N}}>q_k$. Then
     \begin{align*}
         q_{k+1}>\frac{q^{4N+3}}{12C'q_k}>q^{N}>|w|^{N}
     \end{align*}
     which contradicts with the hypothesis on $|w|$.

Case 2: $\ell=\pm\eta_j$ and 
    \begin{align}\label{|eta_j+1|}
        \eta_{j+1}^-<(\eta_j^-)^{r^3}<(2q)^{r^3}.
    \end{align}
    By the definition of nearest oriented affine hyperplane, we have
    \begin{align}\label{align:eta_j+1}
        |a_{j+1}\lambda+b_{j+1}\mu+c_{j+1}|<\frac{1}{q^{4(N+1)}}
    \end{align}
   
    The intersection of $\eta_j$ and $\eta_{j+1}$ is $$\left(\frac{c_{j+1}p+(qm-pn)b_{j+1}}{a_{j+1}p+b_{j+1}q},\frac{c_{j+1}q+(pn-qm)a_{j+1}}{a_{j+1}p+b_{j+1}q}\right)=: \left(\frac{p_{h,1}}{q_h},\frac{p_{h,2}}{q_h}\right) =: \dot{v}_h$$ 
    in lowest terms. So it follows by (\ref{|eta_j+1|}) and $r^3<3/2$ that $q_h\leq a_{j+1}p+b_{j+1}q\leq 2|\eta_{j+1}|q\leq 2^{r^3+1}q^{r^3+1} <q^3<|w|^{3r}<q_{k^\prime}$. Then from (\ref{align:eta_l}) and (\ref{align:eta_j+1}) it follows that 
    \begin{align*}        
        \|(\lambda,\mu)-\dot{v}_h\|_{\infty}<\frac{2\max(|a_{j+1}+q|,|b_{j+1}+p|)}{q^{4(N+1)}|a_{j+1}p+b_{j+1}q|}<\frac{1}{q^{4N+1}}.
    \end{align*}
    With the same argument in case 1, we have
    \begin{align*}
        \|(\lambda,\mu)-\dot{v}_h\| < C^\prime\|(\lambda,\mu)-\dot{v}_h\|_{\infty}<\frac{1}{2(4q^2)^2}<\frac{1}{2q_h^2}
    \end{align*}
    which implies that $\dot{v}_h$ is a best approximation vector by Theorem~\ref{thm:proof of best}. 
    Since
     \begin{align}
         \frac{1}{2q_hq_{h+1}}<\frac{C'}{q^{4N+1}}.
     \end{align}
     we have
    \begin{align*}
         q_{h+1}>\frac{q^{4N+1}}{2C'q_h}>(q^3)^N>(q_h)^N
     \end{align*}
    which implies $h\in l_N$. We follow the previous prove and will get 
    \begin{align*}
        \frac{1}{2q_kq_{k+1}}\leq\frac{1}{2q_hq_{h+1}}<\frac{C'}{q^{4N+1}}.
    \end{align*}
    Since $\alpha>1$, we have $q>|w|>q_{k+1}^{\frac{1}{N}}>q_k$. Then
    \begin{align*}
        q_{k+1}>\frac{q^{4N+1}}{2C'q_k}>q^{N}>|w|^{N}
    \end{align*}
    which is also impossible.  
\end{proof}

\subsubsection{Normal children slits}
Given a slit $w$ let $$\beta=|w|^{r-1}.$$
We aim to show that the definition of a normal slit exhibits a form of recurrence, such that many slits \( w' = w + 2v \in \Delta(w, \alpha, \beta) \) are \( \alpha r \)-normal if \( w \) is an \( \alpha \)-normal slit.

Let us begin by assuming an $\alpha$-normal slit $w$. Consider a loop $v$ satisfying
\begin{equation}\label{def:child}
 |w|^r\le|v|\le2|w|^r \quad\text{ and }\quad |w\times v|<\frac{1}{\alpha},
\end{equation}  
then $w'=w+2v$ define a new slit. If $w'$ is $\alpha r$-normal and satisfies (\ref{def:child}) then it will be called a {\em child} of $w$.
If $w$ is $(\alpha\rho^{N'+1},\beta)$-good, Lemma~{\ref{lem:N'-good=>normal}} together with Lemma~\ref{lem:good:children} will prove that $w'$ is either $\alpha r$-normal or miracle.
Thus, the main task is to bound the number of $w'$ that satisfy (\ref{def:child}) but are not $\alpha r $-normal when $w$  is not  $(\alpha \rho^{N'+1}, \beta)$-good.  
The idea has been addressed in a series of lemmas in \cite{2011Dichotomy} (see \cite[Lemma 7.6--7.10]{2011Dichotomy}).  
We show that these lemmas remain valid in our setting.

Suppose $w$ is $\alpha$-normal. Let $u$ be the convergent of the inverse slope of $w$ with maximum height $|u|< |w|^r$ and $q$ the height of the next convergent. Then we can define $t_1$  
\begin{align}\label{def:t_1}
    |u|=\alpha\rho^{t_1}|w|,
\end{align}
and $t_2$ by
\begin{align}\label{def_t_2}
    q=|w|^{1+(r-1)t_2}.
\end{align}
Because $w$ is $\alpha$-normal, we first note that (\ref{def:normal}) with $t=1$ implies  $1\leq t_1\leq T$. By the definition of $q$, we have $t_2\geq 1$. Then we claim that $t_2\leq t_1$ otherwise (\ref{def:normal}) will fail when $t\in(t_1,t_2)$.

Also for every $w'\in\Delta(w,\alpha,\beta)$, we can denote $u(w')$ be the convergent of the inverse slope of $w'$ with maximum height $|u'|< |w'|^r$ and $q(w')$ the height of the next convergent. Then $t_1(\cdot),t_2(\cdot):\Delta(w,\alpha,\beta)\rightarrow \R$ can be viewed as a function given by
\begin{align}
    |u(w')|=\alpha r\rho^{t_1(w')}|w'|,
    q'=|w'|^{1+(r-1)t_2(w')}.
\end{align}

\begin{lemma}\label{lem:t_1'<N'}
Suppose $q_{k+1}^{1/N}<|w'|<q_{k'}^{1/3r}$.  If $w'$ is neither $\alpha r$-normal nor miracle, then $t_1(w')\leq \min(N',t_2(w'))$.  
\end{lemma}
\begin{proof}
By the definition of $q'$, $t_2(w')\geq 1$. Now we prove by contradiction.
If $t_1(w')>t_2(w')\geq 1$ then (\ref{def:normal}) is satisfied by $|u(w')|$ for all $t\in[1,t_1(w')]$, and by $q'$ for all $t\in[t_1(w'),T]$, contrary to the assumption that $w'$ is not $\alpha r$ normal.   If $t_1(w')> N'$ then Lemma~\ref{lem:N'-good=>normal} applied to the slit $w'$, with $(\alpha r)$ in place of $\alpha$, also implies that $w'$ is $\alpha r$-normal.  
\end{proof}

\begin{lemma}\label{lem:strips}
Let $\bar {t}_1' := \max(t_1(w'),1)$. Then $u(w')$ determines a (nonzero) integer $a\in\mathbb{Z}$ such that 
  $$|\big(w\times u(w')\big)+2a|<\frac{1}{|w|^{r(r-1)\Bar{t}_1'}}.$$  
Moreover, $|a|<2\rho^{N'+1}$.  
\end{lemma}
\begin{proof}
Write $w'=w+2v$ and $v$ satisfies $\frac{|w'\times v|}{|w'||v|}<\frac{1}{\alpha|v|^2}<\frac{1}{2|v|^2}$ which deduces that $v$ is a convergent of $w'$.
Let $v'$ be the next convergent of $w'$ after $v$.  
Since $|u(w')|>|w'|>|v|$ we either have $u(w')=v'$ or $u(w')$ comes after $v'$ in the continued fraction expansion of $w'$. From the property of continued fractions, we know that $v$ and $v'$ form a basis that $u(w')=av'+bv$ for some \emph{nonnegative} integers $a\ge b\ge0$ with $\gcd(a,b)=1$.  
Since $v\times v'=\pm1$ we have 
  $$\big|w'\times u(w')\big| = \big|\big(w\times u(w')\big) + 2\big(v\times u(w')\big)\big| = \big|\big(w\times u(w')\big) \pm 2a\big|.$$  
On the other hand, $$\big|w'\times u(w')\big| 
    < \frac{|w'|}{q'}=\frac{1}{|w'|^{(r-1)t_2'}} < \frac{1}{|w|^{r(r-1)\Bar{t}_1'}}.$$  
This proves the first part.  

By the first inequality in (\ref{ieq:cfc}), $|v'|>\frac{|w'|}{2|w\times v|}$. Then by $|w\times v|<\frac{1}{\alpha}$, we have $|v'|>\frac{\alpha |w'|}{2}$ which implies
$$a<\frac{\big|u(w')\big|}{|v'|} < 2r\rho^{t_1(w')}<2\rho^{N'+1},$$ by Lemma~\ref{lem:t_1'<N'} and since $r<\rho$.  
This proves the second part.  
\end{proof}

Let $w''\in\Delta(w,\alpha,\beta)$ satisfy (\ref{def:child}) and is not $(\alpha r)$-normal and miracle. We also suppose that $|w''|<q_{k'}^{1/3r}$.  
\begin{definition}
If $u(w'')$ determines the same integer $a$ determined by $u(w')$ as in Lemma~\ref{lem:strips}, then we denote that $u(w')$ and $u(w'')$ belong to the same \emph{strip}. 
\end{definition}
Then by Lemma~\ref{lem:strips}, the number of strips is bounded by 
\begin{equation}\label{bound:strips}
  4\rho^{N'+1}.  
\end{equation}

\begin{definition}
Suppose $u(w'),u(w'')$ belong to the same strip, we denote that $u(w')$ and $u(w'')$ lie in the same \emph{cluster} if they differ by a multiple of $u$.
\end{definition}
\begin{lemma}\label{lem:same:cluster}
 Suppose that \footnote{It will be guaranteed by the choice of $q_0$ in \S\ref{s:Init}, so we just admit it in this section.}
$$|w|^{(r-1)^2}\geq 5.$$
Then if $|u(w'')-u(w')|<|w|^r$, they belong to the same cluster.  
\end{lemma}
\begin{proof}
Since $u(w'),u(w'')$ determine the same $a$, Lemma~\ref{lem:strips} and the 
  fact that $\bar{t}_1'\geq 1$ implies 
\begin{equation}\label{ieq:width:strip}
  |w\times\big(u(w'')-u(w')\big)|<\frac{2}{|w|^{r(r-1)}}
\end{equation}
  so that writing $u(w'')-u(w')=d\Bar{u}$ where $d=\gcd\big(u(w'')-u(w')\big)$ we have 
  $$\frac{|w\times(u(w'')-u(w'))|}{|u''-u'||w|} 
        < \frac{2}{|u(w'')-u(w')||w|^{r+(r-1)^2}} \le \frac{1}{2|u(w'')-u(w')|^2},$$ 
  which implies $\Bar{u}$ is a convergent of $w$.  
Since $|\Bar{u}|\le|u(w'')-u(w')|\le|w|^r$, we have $|\Bar{u}|\le|u|$, through the definition of $u$.  
We prove by contradiction and assume that $|\Bar{u}|<|u|$. 
Since $u$ is a convergent of $w$ coming after $\Bar{u}$, 
$$|w\times\Bar{u}|>\frac{|w|}{2|\Bar{u}|}>\frac{|w|}{2|u|}$$ which together with (\ref{ieq:width:strip}) implies 
$$\frac{d|w|}{2|u|} < \frac{2}{|w|^{r(r-1)}}$$ 
so that $$|u|>\frac{d|w|^{r+(r-1)^2}}{4} \ge |w|^r,$$ 
contradicting the definition of $u$.  
We conclude that $\Bar{u}=u$, so that $u(w'),u(w'')$ differ by a multiple of $u$.  
That is, they belong to the same cluster.  
\end{proof}

Pick a representative from each cluster.  To bound the number of clusters we bound the number of representatives.  
Since $|u(w')|=\alpha r\rho^{t_1(w')}|w'| < 5\alpha\rho^{N'+1}|w|^r$ and the difference in height of any two representatives is greater than $|w|^r$, the number of clusters is bounded by (since $\alpha>1$) 
\begin{equation}\label{bound:clusters}
  5\alpha\rho^{N'+1}+1 \le 6\alpha\rho^{N'+1}.  
\end{equation}

We randomly pick $u'$ in each cluster, the number of elements in the cluster can be bounded by the following lemma.

\begin{lemma}\label{lem:cluster:size}
Suppose $t_1(w')\ge t_1-1$ (we assume that $w'$ defines the smallest $t_1$ in this cluster).  
Then the number of elements in the cluster is bounded by 
$$5|w|^{(r-1)-(r-1)^2}.$$ 
\end{lemma}
\begin{proof}
Lemma~\ref{lem:strips} implies for any $u(w'),u(w'')$ in the cluster 
  $$|w\times\big(u(w'')-u(w')\big)|\le\frac{2}{|w|^{r(r-1)\Bar{t}_1'}}$$ 
  where $\Bar{t}_1'$ is the smallest possible within the cluster.  
On the other hand, $$|w\times u| > \frac{|w|}{2q} 
    = \frac{1}{2|w|^{(r-1)t_2}}\ge\frac{1}{2|w|^{(r-1)t_1}}$$ 
  so that 
  $$\frac{|w\times\big(u(w'')-u(w')\big)|}{|w\times u|}<4|w|^{(r-1)(t_1-r\Bar{t}_1')}.$$
By the definition of $u(w'')-u(w')$ is a multiple of $u$. To get the desired bound, using the assumptions $1<r<2$ and $|w|^{r-1}\geq 1$, 
 it remains to show that $$t_1-r\Bar{t}'_1 \le 2-r.$$  
If $t_1(w')>1$, and since $t_1(w')\ge t_1-1$.
  $$t_1-rt_1(w') = \big(t_1-t_1(w')\big)+(1-r)t_1(w')<2-r,$$
whereas if $t_1'\le1$ then 
  $$t_1-r \le t_1(w')+1-r \le 2-r.$$
\end{proof}

We will now combine the above lemmas to demonstrate that, under appropriate assumptions, an $\alpha$-normal slit $w$ will have many children, i.e., $\alpha r$-normal slits $w'$ satisfying the condition in (\ref{def:child}).

\begin{proposition}\label{prop:normal}
 Suppose $w$ is an $\alpha$-normal slit satisfying 
 $$q_{k+1}^{1/N}\le|w|<5|w|^r<q_{k'}^{1/3r}$$ 
where $k,k'$ are consecutive elements of $\ell_N$.  
Suppose further that 
\begin{equation}\label{ieq:normal}
  |w|^{(r-1)^2}\geq \max(480\alpha^2\rho^{N'+3}/c_0,c).
\end{equation} 
Then the number of $w'$ satisfying (\ref{def:child}) that are 
  $\alpha r$-normal is at least 
  $$\frac{c_0|w|^{r-1}}{4\alpha\rho^{N'+1}}.$$
\end{proposition}
\begin{proof}
Let $t_1$ be the parameter associated with the convergent $u$ of $w$ 
  as in (\ref{def:t_1}).  There are two cases.  
If $t_1\ge N'+1$ then $w$ is $(\alpha\rho^{N'+1},|w|^{r-1})$-good, 
  so that Lemma~\ref{lem:good} implies $w$ has at least 
\begin{equation}\label{double}
  \frac{c_0|w|^{r-1}}{\alpha\rho^{N'+1}}
\end{equation} 
 $w'=w+2v$ satisfying (\ref{def:child}).  
While among all the slits at most one is miracle by Lemma~\ref{Lem:danger slit} under the assumption $|w|^{r^2-1}>|w|^{(r-1)^2}\geq c$ .  
Then the number of non-miracle slits is at least $$\frac{c_0|w|^{r-1}}{\alpha\rho^{N'+1}}-1 \geq \frac{c_0|w|^{r-1}}{2\alpha\rho^{N'+1}}.$$
 Moreover, by Lemma~\ref{lem:good:children} each $w'$ constructed is $(\alpha\rho^{N'+1}-\frac12,|w'|^{r-1})$-good.  
Since $$\alpha\rho^{N'+1}-\frac12>\alpha r\rho^{N'},$$ by the choice of $\rho$, every non-miracle $w'$ is a $(\alpha r\rho^{N'},|w'|^{r-1})$-good slit. 
Then we use Lemma~\ref{lem:Uniquly N'-good=>normal} or Lemma~\ref{lem:N'-good=>normal} to imply that such a $w'$ constructed is $\alpha r$ normal except the miracle slit. 

Now consider the case $t_1<N'+1$.  
In this case $w$ is $(\alpha\rho^{t_1},|w|^{r-1})$-good, so that Lemma~\ref{lem:good} implies $w$ has at least 
$$\frac{c_0|w|^{r-1}}{\alpha\rho^{t_1}}-1>\frac{c_0|w|^{r-1}}{2\alpha\rho^{N'+1}}$$ 
$w'$ satisfying  (\ref{def:child}) and not miracle. Moreover, the Lemma~\ref{lem:good:children} implies that each child $w'$ constructed is $(\alpha\rho^{t_1}-\frac12,|w'|^{r-1})$-good, and since $$\alpha\rho^{t_1}-\frac12>\alpha r\rho^{t_1-1},$$ again, by the choice of $\rho$, this means $w'$ is $(\alpha r\rho^{t_1-1},|w'|^{r-1})$-good.  

Moreover, the parameter $t_1(w')$ associated with the convergent $u'$ of each such $w'$ satisfies $t_1(w')\ge t_1-1$.  
Applying Lemmas~\ref{lem:strips}, \ref{lem:same:cluster} and \ref{lem:cluster:size} 
  we conclude the number of $w'$  constructed that are not $\alpha r$-normal 
  is at most the product of the bounds given in (\ref{bound:strips}), (\ref{bound:clusters}), 
  and Lemma~ (\ref{lem:cluster:size}), i.e. $$120\alpha\rho^{2N'+2}|w|^{(r-1)-(r-1)^2},$$ 
  which is at most one fourth of the amount in (\ref{double}) since (\ref{ieq:normal}) holds.  
\end{proof}

\subsection{Liouville construction}\label{s:Liouville}
We turn our focus to the Liouville construction. The key proof presented here relies on the concept of Liouville convergents introduced in Definition~\ref{def:Liouville convergent} and their properties established in Lemma~\ref{Lemma:Liouville convergent}.

Let $u$ be the Liouville convergent of a slit $w=(\lambda+m,\mu+n)$ indexed by $k$ then we have 
\begin{equation}\label{Liouville:convergent}
  (p_{k,1}+mq_k,p_{k,2}+nq_k) = du, \quad \gcd(u)=1 
\end{equation}
  where $$d=d(w,k)=\gcd(p_{k,1}+mq_k,p_{k,2}+nq_k).$$  
A simple calculation on (\ref{Liouville:convergent}) implies the height of $u$ satisfies 
\begin{equation}\label{eq:d|u|}
    \frac{|w|q_k}{2}<d|u|<2|w|q_k.
\end{equation}
Choose $\Tilde{u}\in\mathbb{Z}\times\mathbb{Z}_{>0}$ so that 
  $$|u\times\Tilde{u}|=1 \quad\text{ and }\quad |\Tilde{u}|\le|u|.$$  
Observe that there are exactly $2$ possibilites for $\Tilde{u}$.  

Let $$\Lambda_1(w,k)=\{w+2v : v = \Tilde{u}+au, a\in\mathbb{Z}_{>0}\}$$ consist of children $w+2v$ such that $v$ forms a basis for $\mathbb{Z}^2$ together with $u$, i.e. $\mathbb{Z}^2=\mathbb{Z} u+\mathbb{Z} v$.  
  
The next lemma gives a bound on the cross-product of a parent with a child, which recall, is a necessary estimate in Liouville construction.
  
\begin{lemma}\label{lem:max:area}
Let $u$ be the Liouville convergent of $w$ indexed by $k$. Then $w+2v\in\Lambda_1(w,k)$ for some $|v|<\frac{q_{k+1}^{1/2}}{\Tilde{C}}$ will imply
\begin{equation}\label{ieq:max:area}
  |w\times v| < \frac{2|w|}{|u|} < \frac{4d(w,k)}{q_k} 
\end{equation}
\end{lemma}

\begin{proof}
By (\ref{align:upper bound for liouville convergent}), we have
$$\sin(\angle uw)<\frac{\Tilde{C}}{|u|q_{k+1}^{1/2}}.$$ 
Since $|v|<\frac{q_{k+1}^{1/2}}{C}$ and $|u\times v|=1$ we have 
  $$\sin(\angle uv) = \frac{|u\times v|}{|u||v|} > \frac{\Tilde{C}}{|u|q_{k+1}^{1/2}}$$
  so that $\sin(\angle vw)\le \sin(\angle uv)+\sin(\angle uw)<2\sin(\angle uv)$.  
Therefore, 
  $$|w\times v| = |w||v|\sin(\angle vw)< 2|w||v|\sin(\angle uv) 
                = \frac{2|w|}{|u|}.$$  
\end{proof}

The following lemma encapsulates the key property of the slits constructed using the Liouville construction.  

Let \(d(w, k)\) measure how far \(\frac{p_1 + mq}{p_2 + nq}\) is from being a reduced fraction; specifically, it quantifies the amount of cancellation between the numerator and denominator.  
It is remarkable that, whenever a new slit \(w'\) is constructed by using the Liouville construction, the inequality \(d(w', k) \leq 2\) always holds.
\begin{lemma}\label{lem:gcd}
For any $w'\in\Lambda_1(w,k)$, we have $d(w',k)\le2$.  
Hence, if $|w'|<\frac{q_{k+1}^{1/2}}{\Tilde{C} q_k}$, then the inverse slope of $w'$ 
  has a convergent whose height is between $\frac{q_k|w'|}{4}$ and $2q_k|w'|$.  
\end{lemma}
\begin{proof}
Let $w'=(\lambda+m',\mu+n')$ where $$(m',n')-(m,n)=w'-w=2v.$$  
Now $d'=d(w',k)$ is determined by $d'u'=(p_k+m'q_k,n'q_k)$ 
  for some primitive $u'\in\mathbb{Z}^2$.  
In terms of the basis given by $u$ and $\tilde{u}$ we have 
  $$d'u'=(p_{k,1}+mq_k,p_{k,2}+nq_k)+2q_k(\tilde{u}+au)=(2q_k)\tilde{u}+(2aq_k+d)u.$$  
Note that $d=\gcd(p_{k,1}+mq_k,p_{k,2}+nq_k)$. Since $\gcd(p_{k,1},p_{k,2},q_k)=1$, we have 
  $$d' = \gcd(2aq_k+d,2q_k) = \gcd(d,2q_k) = \gcd(p_{k,1}+mq_k,p_{k,2}+nq_k,2q_k) \le2.$$  
The second statement follows from Lemma~\ref{Lemma:Liouville convergent} and (\ref{eq:d|u|}).  
\end{proof}

Given $r>1$ we let 
  $$\Lambda(w,k)=\{w+2v\in\Lambda_1(w,k) : |w|^r \le |v| \le 2|w|^r \}.$$  
The next lemma gives a lower bound for the number of children constructed in the Liouville construction. 
\begin{lemma}\label{lem:liouville}
If $|w|^{r-1}\ge 2q_k$ then 
\begin{equation}\label{ieq:liouville}
  \#\Lambda(w,k) \ge \frac{|w|^{r-1}}{2q_k}.  
\end{equation}
\end{lemma}
\begin{proof}
Since there are $2$ options for $\Tilde{u}$, the number of slits in $\Lambda(w,k)$ is at least 
$$\#\Lambda(w,k) \ge 2\left[\frac{|w|^r}{|u|}\right]
       \ge \frac{|w|^r}{|u|} \ge \frac{|w|^{r-1}}{2q_k}$$ 
where $|u|\le 2|w|q_k\le|w|^r$ was used in the last two inequalities.  
\end{proof}

\section{Choice of Initial and Transition slits}\label{s:Init}
In \S\ref{s:Lower}, we will build a tree of slits using Diophantine construction and Liouville construction and then associate the tree of slits with a Cantor set \(F(r)\) and express its Hausdorff dimension in terms of the parameters \(r\), \(|w_0|\), \(\delta_j\), and \(\rho_j\). In this section, we will specify these parameters first.  
Moreover, we will establish a connection between the two constructions introduced in \S\ref{s:Diophantine} and \S\ref{s:Liouville} by determining the type of construction to be used at each level for finding the slits of the next level. The levels where Diophantine constructions are applied are referred to as the \emph{Diophantine region}, while those using Liouville constructions are called the \emph{Liouville region}. Typically, the Diophantine region follows the Liouville region, and it is necessary to decide when the slits transition from the Liouville region to the Diophantine region.  

The initial Diophantine region requires normal slits to apply Proposition~\ref{prop:normal}, and Lemma~\ref{lem:N'-good=>normal} plays a crucial role in ensuring the existence of normal slits. The conditions in Lemma~\ref{lem:N'-good=>normal} require us to avoid a special case known as \textit{miracle slits}. To achieve this, we define \(j^D_k\) as the transition point between the Liouville and Diophantine regions, ensuring that no slit at Level \(j^D_k\) is a miracle slit.

Given $\eps>0$, we adopt the notation from \S\ref{s: two construction} and choose $1< r <r^3 <\frac{3}{2}$ satisfying 
  $$\frac{1}{1+r}>\frac{1}{2}-\eps.$$ 
Then we choose $\delta$ such that
\begin{equation*}
  \frac{1-\delta}{1+r+2\delta}>\frac{1}{2}-\eps.  
\end{equation*}
It will be convenient to set $$M:=\frac{1}{r-1}>1,$$
\begin{equation}\label{def:M'}
  M'=\max(3M^2,Mr/\delta),
\end{equation}
\begin{equation}\label{def:N}
  N=4M'r^5
\end{equation}
\color{black}
and let $N'$ be given by (\ref{aligh: how deep for a given convergent}).

\subsection{Initial slit}\label{ss:w_0}
Now because $\ell_N$ is a infinite set that we can choose $k_0\in\ell_N$ large enough so that\footnote{When $(\lambda,\mu)$ is uniquely rational, we additionally choose $q_{k_0} > \max(|a|, |b|)$, as required in Lemma~\ref{lem:Uniquly N'-good=>normal}. } 
\begin{equation}\label{def:k_0}
  q_{k_0} > \max\left(c^M,60c_0^{-1}\rho^{N'+3},4\rho^{N'}(\log_r(3M')+4),2^7\rho^{N'},48C'\right)
\end{equation}
where $c=5C'^5>5$ is given in Lemma~\ref{Lem:danger slit}.

We choose $w_0$ satisfying the conditions of the following lemma and let it be fixed for the rest of this paper. This lemma benefits us to always construct the tree of slits beginning with a Liouville construction. It is the unique slit of level $0$.  
\begin{lemma}\label{lem:init:slit}
There is a slit $w_0\in V_2^+$ such that $d(w_0,k_0)\le2$ and 
\begin{equation}\label{ieq:init:slit}
   q_{k_0}^{M'} \le |w_0| < q_{k_0}^{M'r}.  
\end{equation}
\end{lemma}
\begin{proof}
We first note that the height of interval $[q_{k_0}^{M'}, q_{k_0}^{M'r})$
$$q_{k_0}^{M'r}-q_{k_0}^{M'}=q_{k_0}^{M'}(q_{k_0}^{M'/M}-1)\geq q_{k_0}^{M'}(q_{k_0}^{3M}-1)>q_{k_0}^{M'}.$$
Let $w\in V_2^+$ be any slit of height $|w|<q_{k_0}/4$ and $u$ be the Liouville convergent of $w$ indexed by $k_0$. Then its height $|u|\le 2q_{k_0}|w|\le q_{k_0}^2/2$. 
Since two adjacent elements in $\Lambda_1(w,k_0)$ differ by $2u$, 
and $q_{k_0}^{M'}>q_{k_0}^{2}$, one can choose a slit $w_0\in\Lambda_1(w,k_0)$
satisfying (\ref{ieq:init:slit}). By Lemma~\ref{lem:gcd}, we have $d(w_0,k_0)\le2$.
\end{proof}

Then the choice of $k_0$ in the first term of (\ref{def:k_0}) gives us a lower bound on the height of $w_0$, 
\begin{equation}\label{ieq:w_0>5} 
  |w_0|^{(r-1)^2} \ge q_{k_0}^{M'/M^2 }> q_{k_0}^{1/M} > c >5.
\end{equation}

\subsection{Transition slits}\label{ss:indices}
Next, we specify the type of construction to be applied at each level \(j \geq 0\) for generating the slits of the subsequent level 
(the same type of construction is applied uniformly to all slits within a given level).

We define indices \(j^A_k\) for each \(k \in \ell_N\) with \(k \geq k_0\) and \(A \in \{B, C, D\}\) such that, whenever \(k < k'\) are consecutive elements of \(\ell_N\), the following inequality holds (see Lemma~\ref{lem:indices}(i) below):
\[
j^B_k < j^C_k < j^D_k < j^B_{k'}.
\]
For \(j^C_k \leq j < j^D_k\), we apply the construction described in \S\ref{s:Liouville}, 
while for all other \(j\), we use the techniques described in \S\ref{s:Diophantine}.  
The choice of \(j^C_k\) and \(j^D_k\) must ensure that the slits at level \(j\) for \(j^C_k \leq j \leq j^D_k\) satisfy the conditions in Lemma~\ref{lem:gcd} and Lemma~\ref{lem:liouville}. Meanwhile, the choice of \(j^D_k\) and \(j^B_{k'}\) must guarantee that the slits at levels \(j^D_k \leq j \leq j^B_{k'}\) satisfy the conditions in Proposition~\ref{prop:normal}. 
The precise implementation of these constructions is described in the next section.
  
Firstly, we introduce our strategy for choosing the indices \(j^A_k\). Using the two types of constructions, we know that 
\begin{align}\label{ieq:|w_{j+1}|}
    |w_j|^r<|w_{j+1}|<5|w_j|^r
\end{align}
which implies that the logarithm of the height sequence of \(w_j\), taken to the base \(|w_0|\), approximately increases in a geometric progression with ratio \(r\).  

Define \(H_0 = \{|w_0|\}\), and for levels \(j > 0\), the heights of all slits at level \(j\) are contained within the interval
\[
H_j = \left[|w_0|^{r^j}, 5^{\frac{r^j-1}{r-1}}|w_0|^{r^j}\right].
\]
The next lemma demonstrates that the sequence of intervals \(H_j\) forms a collection of closed, disjoint intervals in \(\mathbb{R}\).
\begin{lemma}\label{lem:H_j}
 For all $j\ge0$ 
\begin{equation}\label{ieq:H_j}
  \sup H_j < \inf H_{j+1} = (\inf H_j)^r.
\end{equation}
\end{lemma}
\begin{proof}
The condition $\sup H_j<\inf H_{j+1}$ is equivalent to $$5^{\frac{r^j-1}{r-1}}<|w_0|^{r^j(r-1)},$$ 
  which is implied by $$5^{r^j}<|w_0|^{(r-1)^2r^j},$$ which in turn is implied by (\ref{ieq:w_0>5}).  
\end{proof}

The choice of the indices $j^A_k$ will depend on the position of $H_j$ relative to 
  that of the following intervals: $$I^C_k=\left[q_k^{M'},q_{k+1}^{1/3r}\right),\qquad\text{and}\qquad 
      I^D_k=\left[q_{k+1}^{1/3r^5},q_{k'}^{1/3r}\right).$$  

Here $k,k'$ are successive elements in $\ell_N$. 
we will choose 
These intervals overlap nontrivially and the overlap cannot be too small in the sense that there are at least three consecutive $H_j$'s contained in it.  
\begin{lemma}\label{lem:IKCD3}
 For any $k\in\ell_N$ with $k\ge k_0$ 
\begin{equation}\label{ieq:IKCD3}
  \#\{j: H_j \subset I^C_k\cap I^D_k\}\ge3.  
\end{equation}
\end{lemma}
\begin{proof}
We first note that $\inf H_{j+1}=(\inf H_j)^r$ implies the logarithm of $\inf H_j$ to base $|w_0|$ is a geometric series of $r$. We claim that the fact in  Lemma~\ref{lem:H_j} implies 
  $$\#\{j\ge0: H_j\subset\left[q_{k+1}^{1/3r^5},q_{k+1}^{1/3r}\right)\} 
       \ge \left\lfloor\log_r(r^4)\right\rfloor -1=3.$$   

To show this, it suffices to check that
\[
q_{k+1}^{1/(3r^4)} > q_k^{N/(3r^4)} > q_k^{M'r} > |w_0|,
\]
which is ensured by the assumptions \( N = 4M'r^5 \) and \( q_{k+1} > q_k^N \).

\end{proof}

By virtue of the fact that the quantity in (\ref{ieq:IKCD3}) is at least one, we can now give two equivalent definitions of the index $j^A_k$. When defining $j^D_k$, we want to ensure that there are no miracle slits at the last level of the Liouville region so that Lemma~\ref{lem:N'-good=>normal} can be applied to all of them.   
\begin{definition}
For $k<k'$ consecutive elements of $\ell_N$ with $k\ge k_0$, let 
\begin{align*}
  j^C_k &= \min\{j:H_j\subset I^C_k\} = \min\{j:\inf H_j\ge q_k^{M'}\}\\
  j^B_{k'} &= \max\{j:H_j\subset I^D_k\} = \max\{j:\sup H_j<q_{k'}^{1/3r}\} 
\end{align*} 
Note that $j^C_{k_0}=0$ and that $j^B_{k_0}$ is not defined.  

Let $$j_k = \max\{j:H_{j+1}\subset I^C_k\} = \max\{j:\sup H_{j+1}<q_{k+1}^{1/3r}\}.$$ We define the slits in levels $j_k^C, j_k^C+1, \dots, j_k-1$ using the Liouville construction.  We apply the Liouville construction once more to all the slits of level $j_k-1$, i.e. we use the construction described in \S\ref{s:Liouville} with $\delta=\frac{4}{q_k}$ to find slits.  If no miracle slits appear in this manner, we let $j_k^D=j_k$ determine the beginning of the Diophantine region; otherwise, we let $j_k^D=j_k-1$. The next lemma ensures that none of the slits of level $j_k^D$ are miracle.  
\end{definition}

\begin{lemma}\label{2 levels with 1 danger slit}
If the Liouville construction applied to any slit of level $j_k-1$ results a miracle slit, then none of the slits of level $j_k-1$ can be miracle.  
\end{lemma}
\begin{proof}
    Suppose on the contrary that level $j_k-1$ and $j_k$ contain miracle slits $w$ and $w^\prime$, respectively. Then the definition of miracle slit implies
    $$\eta_j^->\frac{|w|}{C^\prime}>\frac{|w_0|^{r^{j_k-1}}}{C^\prime}$$
    and
    $$\eta_{j^\prime}^-\geq\eta_{j+1}^->(\eta_j^-)^{r^3}>\frac{|w_0|^{r^{j_k+2}}}{{C^\prime}^{r^3}}.$$
    Then
    $$|w^\prime|>{(\frac{{\eta_{j^\prime}}^-}{C^\prime}})^{1/r}>\frac{|w_0|^{r^{j_k+1}}}{{C^\prime}^{r^2+1}}>5^{\frac{r^{j_k}-1}{r-1}}|w_0|^{r^{j_k}}$$
    which the last inequality is given by $|w_0|^{(r-1)^2}>c=5{C^\prime}^5.$
    Then $w^\prime$ is not in level $j_k$ by Lemma~\ref{lem:H_j}.
\end{proof}

The main facts about these indices are expressed in the next two lemmas.  
\begin{lemma}\label{lem:indices}
For any $k\in \ell_N$, $k\ge k_0$ 
\begin{itemize}
   \item[(i)] $j^B_k < j^C_k < j^D_k \leq j^B_{k'};$
  \item[(ii)] $j^C_k\le j^B_k + \log_r(3M')+ 4.$  
\end{itemize}
\end{lemma}
\begin{proof}
For (i) we note that 
  $$\inf H_{j_k^B} \le \sup H_{j_k^B} \le q_k^{1/3r} < q_k^{M'}$$ 
  so the first inequality follows by the (second) definition of $j_k^C$.  
From the first definitions of $j^C_k$ and $j^D_k$, we see that the second inequality is a consequence of Lemma~\ref{lem:IKCD3} and definition of $j^D_k$.  
The third inequality follows by comparing the second definitions of ${j_k}$ and $j^B_{k'}$ and noting that $q_{k'}\ge q_{k+1}$.  

For (ii) first note that 
  $$\inf H_{j^B_k} = \left(\inf H_{j^B_k+1}\right)^{1/r} 
       \ge\left(\sup H_{j^B_k+1}\right)^{1/r^2} \ge q_k^{1/3r^3}$$ 
  by Lemma~\ref{lem:H_j} and the second definition of $j^B_k$.  
Thus, we have 
  $$\inf H_{j^B_k+n} = \left(\inf H_{j^B_k}\right)^{r^n} 
        \ge q_k^{r^{n-3}/3} \ge q_k^{M'}$$ 
  where $n=\lceil\log_r(3M')+4\rceil$.  
The second definition of $j^C_k$ now implies $j^C_k\le j^B_k+n\le j^B_k+\log_r(3M')+4$.  
\end{proof}

\begin{lemma}\label{lem:range}
Let $\Tilde{C}$ be given by Lemma~\ref{Lemma:Liouville convergent}. For any slit $w$ of level $j$ we have 
\begin{itemize}
  \item[(i)] $\ds \quad j^C_k \le j \le j^D_k \impl |w|\in I^C_k \impl (2q_k)^M\le|w|<\frac{q_{k+1}^{1/2}}{\Tilde{C} q_k};$
  \item[(ii)] $\ds \quad j^D_k \le j \le j^B_{k'} \impl |w|\in I^D_k \impl q_{k+1}^{1/N'}\le|w|<q_{k'}^{1/3r}.$
\end{itemize}
\end{lemma}
\begin{proof}
By definition, $\inf H_{j^C_k}\ge q_k^{M'}$ and $\sup H_{j^D_k} <q_{k+1}^{1/3r}$, 
  giving the first implication in (i).  
Since $N>4M'\geq12M^2$, we have $$q_{k+1}^{1/2-1/3r}>q_k^{N/6}\ge q_k^{2M^2}>\Tilde{C}q_k$$  
  so that $q_{k+1}^{1/3r}<\frac{q_{k+1}^{1/2}}{\Tilde{C} q_k}$.  
This, together with $q_k^{M'}\ge q_k^{2M}\ge (2q_k)^M$, implies the second implication in (i).  

For (ii) note that (\ref{ieq:IKCD3}) implies 
$H_{j^D_k}\subset I^C_k\cap I^D_k$, giving the first implication, while the second implication follows from $N'>3r^5$.  
\end{proof}

\section{Hausdorff dimension 1/2}\label{s:Lower}
In this section, we first building the slit tree in \S\ref{s:Tree}, then we will associate the tree of slits with a Cantor set and prove that it has Hausdorff dimension \( \frac{1}{2} \).
\subsection{Tree of slits}\label{s:Tree}
We describe precisely the procedure for constructing the slits of level \(j+1\) from those of level \(j\). 
As discussed in \S\ref{s:Init}, this process is classified into three types based on the indices \(j^A_k\), where \(A \in \{B, C, D\}\). 
The conditions on the heights of slits required for the Diophantine and  Liouville constructions, as detailed in \S\ref{s:Diophantine} and \S\ref{s:Liouville}, are resolved by Lemma~\ref{lem:range}. 

Additionally, another set of conditions concerns the continued fraction expansions of the inverse slopes of slit directions. 
The necessity of these conditions is made clear by the Lemma~\ref{lem:good}, which serves as a fundamental tool to determine whether a slit will produce a significant number of children.

For the levels between consecutive indices of the form $j^A_k$, we have guaranteed these conditions when we introduce the constructions in  \S\ref{s: two construction}.  
Special attention is needed to check the relevant hypotheses of the second kind for the levels $j^A_k,k\in{B,C,D}$ when the type of construction used to find the slits of the next level changes.  

The parameters \(\delta_j\) and \(\rho_j\) are also specified in this section.  
At each step, we verify that the choice of \(\delta_j\) and \(\rho_j\) ensures that all cross-products of the slits at the level \(j\) with their children are less than \(\delta_j\), while the number of children is at least \(\rho_j|w|^{r-1}\delta_j\).  

We use \(k < k'\) to denote consecutive elements of \(\ell_N\), with \(k \geq k_0\).  
If \(k > k_0\), then \(\Tilde{k}\) denotes the element of \(\ell_N\) immediately preceding \(k\).

\subsubsection{Liouville region}\label{ss:Liouville}
For the levels $j$ satisfying $j^C_k\le j<j^D_k$, the slits of level $j+1$ will be constructed by applying Lemma~\ref{lem:liouville} to all slits of level $j$.  

For the initial slit $w_0$ given by Lemma~\ref{lem:init:slit}, and every $w'\in \Lambda(w_0,k)$ we can apply Lemma~\ref{lem:max:area} to show that the cross-product of $w_0$ with $w'$ is less than $4/q_{k_0}$. Lemma~\ref{lem:liouville} implies the number children is at least $|w|^{r-1}/2q_{k_0}$.  Therefore, we set $$\delta_0 = \frac{4}{q_{k_0}} \quad\text{ and }\quad \rho_0=\frac18.$$  

For the levels $j^C_k<j<j^D_k$, we set $$\delta_j = \frac{4}{q_k} \quad\text{ and }\quad \rho_j=\frac18.$$  

\begin{lemma}\label{lem:LR}
For $j^C_k<j\le j^D_k$, every slit $w$ of level $j$ satisfies $d(w,k)\le2$.  
Moreover, if $j<j^D_k$ then the cross-products of each slit of level $j$ with its children are less than $\delta_j$ 
  and the number of children is at least $\rho_j|w|^{r-1}\delta_j$.  
\end{lemma}
\begin{proof}
Since all slits of level $j$ were obtained via the Liouville construction, the first part follows from the first assertion of Lemma~\ref{lem:gcd}.  
Suppose $w$ is a slit of level $j$ with $j^C_k<j<j^D_k$.  
Lemma~\ref{lem:max:area} now implies the cross-products of $w$ with 
its children are less than $4/q_k$, and the number of children is at 
least $|w|^{r-1}/2q_k$, by Lemma~\ref{lem:liouville}.  
\end{proof}

The next lemma shows that the slits $w$ lying in the level $j^D_k$ will share a property of normal slit. Let $$\alpha_k=\frac{q_k}{4\rho^{N'}}.$$  
\begin{lemma}\label{lem:LR->DR}
Every slit of level $j^D_k$ is $\alpha_k$-normal.  
\end{lemma}
\begin{proof}
Let $w$ be a slit of level $j^D_k$.  Since $H_{j^D_k}\subset I^C_k$, we have 
  $$|w|^{r-1}=|w|^{\frac{1}{M}}>|w|^{\frac{3}{M'}}\geq 2q_k.$$  
By Lemma~\ref{lem:LR}, we have $d(w,k)\le2$ and since $w$ was obtained via the Liouville construction,  Lemma~\ref{lem:gcd} implies the inverse slope of $w$ has a convergent with height between $q_k|w|/4$ and $2q_k|w|$, or, by the above, between $\alpha_k\rho^{N'}|w|$ and $|w|^r$.  
This means $w$ is $(\alpha_k\rho^{N'},|w|^{r-1})$-good. Because all slits at level $j^D_k$ are not miracles (the definition of $j^D_k$), then $w$ is $\alpha_k$-normal according to Lemma~\ref{lem:N'-good=>normal}.  
\end{proof}

\subsubsection{Diophantine region}\label{ss:Diophantine}
For the levels $j$ satisfying $j^D_k\le j<j^B_{k'}$, we apply Proposition~\ref{prop:normal} with the parameter $\alpha=\alpha_kr^{j-j^D_k}$ to all slits $w$ of level $j$, then the slits of level $j+1$ are consist of all the $\alpha r$-normal children.

For the levels $j^D_k\le j<j^B_{k'}$, we set 
$$\delta_j = \frac{4\rho^{N'}}{q_kr^{j-j^D_k}} \quad\text{ and }\quad 
        \rho_j = \frac{c_0}{4\rho^{N'+1}}.$$  

\begin{lemma}\label{lem:DR}
For $j^D_k\le j\le j^B_{k'}$, every slit $w$ of level $j$ is $\alpha_kr^{j-j^D_k}$-normal.  
Morevover, if $j<j^B_{k'}$ then the cross-products of each slit of level $j$ with its children are less than $\delta_j$ and the number of children is at least $\rho_j|w|^{r-1}\delta_j$.  
\end{lemma}
\begin{proof}
The case $j=j^D_k$ of the first assertion follows from Lemma~\ref{lem:LR->DR} while the remaining cases follow from Proposition~\ref{prop:normal}.  

When we apply Proposition~\ref{prop:normal}, we should firstly verify the inequality (\ref{ieq:normal}) holds, i.e. if 
\begin{equation}\label{ieq:normal:2}
  60q_k^2r^{2(j-j^D_k)}\rho^{N'+3}\le c_0|w|^{(r-1)^2} 
\end{equation}
and
\begin{equation}
  |w|^{(r-1)^2}>c.
\end{equation}
The second condition is given by (\ref{ieq:w_0>5}).
We also note that it is enough to check (\ref{ieq:normal:2}) in case $j=j^D_k$ since the left-hand side increases by a factor $r^2$ as $j$ increases by one, while the right-hand side increases by a factor $|w|^{(r-1)^3}>q_{k_0}^{3(r-1)}>5>r^2$.  
Moreover, since $|w|^{(r-1)^2}>q_k^3$, (\ref{ieq:normal:2}) in the case $j=j^D_k$ follows from $60\rho^{N'+3}<c_0q_{k_0}$, which is guaranteed by the third term in (\ref{def:k_0}).  
  
Then we apply Proposition~\ref{prop:normal} to an $\alpha$-normal slit, the cross-products with its children are less than $1/\alpha$, which is $\delta_j$ if $\alpha=\alpha_kr^{j-j^D_k}$.  
The number of children is at least $$\frac{c_0|w|^{r-1}}{4\alpha\rho^{N'+1}} 
   = \rho_j|w|^{r-1}\delta_j.$$ 
\end{proof}

\begin{lemma}\label{lem:DR->BR}
Every slit $w$ of level $j^B_{k'}$ is $(\alpha_k,|w|^{r-1})$-good.  
\end{lemma}
\begin{proof}
Let $w$ be a slit of level $j^B_{k'}$.  Lemma~\ref{lem:DR} implies 
  that $w$ is $\alpha$-normal for some $\alpha>\alpha_k$.  
By the case $t=1$ in the definition of normality, this means $w$ is 
  $(\alpha,|w|^{r-1})$-good, i.e. its inverse slope has a convergent whose height is between $\alpha|w|$ and $|w|^r$.  
Since $\alpha>\alpha_k$ the height of this convergent is between 
  $\alpha_k|w|$ and $|w|^r$.  
Hence, $w$ is $(\alpha_k,|w|^{r-1})$-good.  
\end{proof}

\subsubsection{Bounded region}\label{ss:Bounded}
For the levels $j$ satisfying $j^B_k\le j< j^C_k, k>k_0$, we will require the property of  $(\alpha,\beta)$-good instead of normal slits because the length of this region (See (ii) in Lemma~\ref{lem:indices}) is small relative to the decrease of $\alpha$ in Lemma~\ref{lem:good}.

The slits of level $j+1$ will be constructed by applying Lemma~\ref{lem:good} to all slits $w$ of level $j$ with the parameters 
\begin{equation}\label{alpha:beta}
  \alpha=\alpha_{\Tilde{k}}-\frac{j-j^B_k}{2} \quad\text{ and }\quad \beta=|w|^{r-1}.  
\end{equation}

For the levels $j^B_k\le j< j^C_k, k>k_0$, we set $$\delta_j = \frac{4\rho^{N'}}{q_{\Tilde{k}}} \quad\text{ and } \quad 
        \rho_j = \frac{c_0}{2}.$$  

\begin{lemma}\label{lem:BR}
For $j^B_k\le j\le j^C_k$, every slit $w$ of level $j$ is $(\alpha_{\Tilde{k}}/2,|w|^{r-1})$-good.  
Moreover, if $j<j^C_k$ then the cross-products of each slit of level $j$ with its children are less than $\delta_j$ 
  and the number of children is at least $\rho_j|w|^{r-1}\delta_j$.  
\end{lemma}
\begin{proof}
Firstly, we note that every slit $w$ of level $j$ is $(\alpha,\beta)$-good, where
$\alpha$ and $\beta$ are the parameters given in (\ref{alpha:beta}).  
Indeed, for $j=j^B_k$ this follows from Lemma~\ref{lem:DR->BR} while for $j^B_k<j\le j^C_k$ it follows from Lemma~\ref{lem:good:children}.  
Lemma~\ref{lem:indices}.ii and the third relation in (\ref{def:k_0}) imply 
  $$j-j^B_k\le j^C_k-j^B_k\le\log_r(3M')+4\le \alpha_{\Tilde{k}}$$ 
  from which we see that the first assertion holds.  

For children constructed using the Lemma~\ref{lem:good} applied to a $(\alpha,\beta)$-good slit, the cross-products are less than $1/\alpha$, 
which is $<\delta_j$, since $\alpha\geq\alpha_{\Tilde{k}}/2$.  
And since $\alpha\le\alpha_{\Tilde{k}}$, the number of children is at least 
  $$\frac{c_0|w|^{r-1}}{\alpha_{\Tilde{k}}} = \rho_j|w|^{r-1}\delta_j$$ 
  giving the second assertion.  
\end{proof}

Finally, for the levels $j=j^C_k$ with $k>k_0$, we set 
  $$\delta_j = \frac{8\rho^{N'}}{q_{\Tilde{k}}} \quad\text{ and } \quad 
        \rho_j = \frac{q_{\Tilde{k}}}{16\rho^{N'}q_k}.$$  
\begin{lemma}\label{lem:BR->LR}
For any slit $w$ of level $j=j^C_k$ with $k>k_0$, the cross-products of $w$ with its children are less than $\delta_j$ and the number of children is at least $\rho_j|w|^{r-1}\delta_j$.  
\end{lemma}
\begin{proof}
Suppose $w$ is a slit of level $j^C_k$ with $k>k_0$.  The case $j=j^C_k$ of Lemma~\ref{lem:BR} implies $w$ is $(\alpha_{\Tilde{k}}/2,|w|^{r-1})$-good.  
Let $u$ be the Liouville convergent of $w$ indexed by $k$.  
By Lemma~\ref{Lemma:Liouville convergent} the height $q'$ of the next convergent is 
  $$q'>\frac{q_{k+1}^{1/2}}{\Tilde{C}} > q_{k+1}^{1/3r} > \left(\sup H_{j^C_k}\right)^r \ge |w|^r.$$  
Since $w$ is $(\alpha_{\Tilde{k}}/2,|w|^{r-1})$-good, we must have 
  $|u|\ge\alpha_{\Tilde{k}}|w|/2$ so that, by Lemma~\ref{lem:max:area} 
  the cross-products of $w$ with its children are $$<
   \frac{2|w|}{|u|} \le \frac{4}{\alpha_{\Tilde{k}}} = \frac{8\rho^{N'}}{q_{\Tilde{k}}}.$$
By Lemma~\ref{lem:liouville}, the number of children is at least 
  $|w|^{r-1}/2q_k = \rho_j|w|^{r-1}\delta_j$.  
\end{proof}

The construction of the tree of slits is now complete. In the following, we demonstrate how a Cantor set can be defined using this tree of slits.

\subsection{The Cantor set arising from the tree of slits}\label{ss:Cantor}

For a slit \(w\), let \(I(w)\) denote the interval of length  
\[
\diam I(w) = \frac{4}{|w|^{r+1}}
\]  
centered at the inverse slope of the direction of \(w\). The following lemma provides estimates for the sizes of these intervals and the gaps between them.
\begin{lemma}\cite[Lemma 5.1]{2011Dichotomy}\label{lem:gaps}
Assume $|w_0|^{r(r-1)}\ge64$ and $\delta_j<\frac{1}{16}$.  
Let $w_{j+1}$ be a child of a slit $w_j$ of level $j$.  Then 
\begin{itemize}
\item $I(w_{j+1})\subset I(w_j)$, and 
\item if $w'_{j+1}$ is another child of $w_j$, then 
  $\dist(I(w_{j+1}),I(w'_{j+1})) \ge \frac{1}{16|w_j|^{2r}}.$ 
\end{itemize}
\end{lemma}
Let $$F_j=\bigcup_w I(w)$$ where the union is taken over all slits of level $j$.
Then the cantor set associated with the tree of slits will be defined as
$$
F=\bigcap_{j \geq 0} F_j.
$$
From (\ref{ieq:|w_{j+1}|}) we have 
\begin{equation}\label{ieq:|w_j|}
  |w_0|^{r^j} \le |w_j| \le 5^{\frac{r^j-1}{r-1}}|w_0|^{r^j},
\end{equation}
so that the number of children is at least 
\begin{equation}\label{def:m_j}
  m_j := \rho_j\delta_j|w_0|^{r^j(r-1)} 
\end{equation}
  while the smallest gap between the associated intervals is at least 
\begin{equation*}\label{def:eps_j}
   \epsilon_j := \frac{1}{16\cdot5^{2r\frac{r^j-1}{r-1}}|w_0|^{2r^{j+1}}}, 
\end{equation*}
by Lemma~\ref{lem:gaps}.  
Then Falconer's lower bound estimate is 
\begin{equation*}
\Hdim F \ge \liminf_j \frac{\log (m_0\cdots m_j)}{-\log m_{j+1}\epsilon_{j+1}}.  
\end{equation*} 
If $\lim_{j\to\infty}m_j\epsilon_j=0$, as is necessarily the case if the 
  length of the longest interval in $F_j$ tends to zero as $j\to\infty$, 
  then 
\begin{align}\label{eq:Falc}
      \Hdim F \ge \liminf_j d_j
\end{align}
where
\begin{align}
    d_j 
    \notag&= \frac{\log m_j}{-\frac{m_{j+1}\epsilon_{j+1}}{m_j\epsilon_j}}\\
     \notag &= \frac{r^j(r-1)\log|w_0| + \log(\rho_j\delta_j)}
               {r^j(r^2-1)\log |w_0| + 2r^{j+1}\log5 -\log(\rho_{j+1}\delta_{j+1}/\rho_j\delta_j)}\\
     &= \frac{1 - \frac{-\log(\rho_j\delta_j)}{r^j(r-1)\log|w_0|}}
              {1+r + \frac{2r\log5}{(r-1)\log|w_0|} 
                   + \frac{\log(\rho_j\delta_j/\rho_{j+1}\delta_{j+1})}{r^j(r-1)\log|w_0|}}\label{eq:local}.  
\end{align}
Now making $d_j$ close to $\tfrac12$ will mean making $r$ close to $1$ 
  and making the terms 
\begin{equation}\label{eq:num}
  \frac{-\log(\rho_j\delta_j)}{r^j(r-1)\log|w_0|}
\end{equation}
  and 
\begin{equation}\label{eq:den}
  \frac{2r\log5}{(r-1)\log|w_0|} + 
  \frac{\log(\rho_j\delta_j/\rho_{j+1}\delta_{j+1})}{r^j(r-1)\log|w_0|}
\end{equation}
small. 
In \S\ref{s:Tree}, we have selected \(\delta_j\) and \(\rho_j\) which satisfied each step of construction.  
Conditions \(\delta_j < \frac{1}{16}\), as required by Lemma~\ref{lem:gaps}, and \(m_j \geq 2\), as required by Falconer's estimate, were verified in \S\ref{s:Tree}, along with the fact that \(|w_0|\) can be chosen sufficiently large to ensure that \(d_j\) is close to \(\tfrac{1}{2}\).

First, we verify the hypotheses needed for Falconer's estimate.  
Recall the definition $m_j=\rho_j|w_0|^{r^j(r-1)}\delta_j$ in (\ref{def:m_j}).  
\begin{lemma}
$\delta_j<\tfrac{1}{16}$ and $m_j\ge2$ for $j\ge0$.  
\end{lemma}
\begin{proof}
By the forth relation in (\ref{def:k_0}) we see that $\delta_j\le\frac{8\rho^{N'}}{q_{k_0}} < \frac{1}{2^4}$.  
Then we consider two regions divided by the methods of construction.
When $j^C_k\le j<j^D_k$ we have $\rho_j\delta_j=1/2q_k$ and since $|w_0|^{r^{j^C_k}}\ge q_k^{M'}$, by the definition of $j^C_k$, we have $|w_0|^{r^j(r-1)}\ge q_k^{M'/M}\ge q_k^{3M}$. Then we have
 $$m_j=\frac{q_k^{3M-1}}{2}\ge\frac{2\rho}{c_0}>2$$ by the second relation in (\ref{def:k_0}).
For $j^D_k\le j < j^C_{k'}$ the expression $|w_0|^{r^j(r-1)}$ increases 
  faster than $\rho_j\delta_j$ decreases, so it is enough to check 
  the case $j=j^D_k$, for which, by the above, we have 
  $m_j\ge\frac{c_0}{\rho}m_{j-1}\ge2$.  
\end{proof}

Next, we calculate the lower bound on the Hausdorff dimension of $F$.  
The next lemma shows that it is close to $\tfrac12$ by the choices made in \S\ref{s:Init}.  

\begin{lemma}\label{lem:d_j}
$\liminf_{j\to\infty} d_j>\frac12-\eps.$  
\end{lemma}
\begin{proof}
By (\ref{eq:local}) it is enough to show that the term (\ref{eq:num}) and both terms in (\ref{eq:den}) are bounded by $\delta$.  
Using the choice of $M'$ in (\ref{def:M'}), it would be enough to show that each term is bounded by $\frac{Mr}{M'}$ for all large enough $j$.   

We first consider the expression (\ref{eq:den}). Using (\ref{ieq:w_0>5}) we see that the first term in (\ref{eq:den}) satisfies 
  $$\frac{2r\log5}{(r-1)\log|w_0|}\le\frac{2r}{M}.$$  
Then by the choice of $\rho_j$ and $\delta_j$ in \S\ref{s:Tree}, we see that for $j\neq j^C_{k-1}$ we have 
  $$\frac{\rho_j\delta_j}{\rho_{j+1}\delta_{j+1}} \in \left\{1,\frac{\rho}{c_0},r, 
      \frac{1}{2\rho^{N'+1} r^{j^B_{k'}-j^D_k-1}},  \right\}$$ 
while for $j=j^C_{k-1}$ we have 
  $$\frac{\rho_j\delta_j}{\rho_{j+1}\delta_{j+1}} = \frac{2c_0\rho^{N'}q_k}{q_{\Tilde{k}}}.$$  
Then, in the second case, we have \footnote{Here $A_j\lesssim B_j$ is an abbreviation for $\liminf A_j\le\liminf B_j$. }$$\frac{\log(\rho_j\delta_j/\rho_{j+1}\delta_{j+1})}{r^j(r-1)\log|w_0|} \le 
    M\left(\frac{\log q_k+\log 2c_0\rho^{N'}}{r^{j^C_k-1}\log|w_0|}\right) \lesssim \frac{Mr}{M'}$$ 
since $\inf H_{j^C_k}\in I^C_k$; in the first case the left hand side above is $\lesssim0$.  

We now turn to the expression (\ref{eq:num}). 
For $j^C_k\le j<j^D_k$ we have  
  $$\frac{-\log(\rho_j\delta_j)}{r^j(r-1)\log|w_0|} 
     \le M\left(\frac{\log q_k}{r^{j^C_k}\log|w_0|}\right)
     \le \frac{M}{M'}.$$  
Next consider $j^D_k\le j<j^B_{k'}$.  Using $jr^{-j}\log r\le1$, 
  we have 
\begin{align*}
  \frac{-\log(\rho_j\delta_j)}{r^j(r-1)\log|w_0|} 
  &\le M\left(\frac{\log q_k+(j-j^D_i)\log r+\log(\rho/c_0)}{r^j\log|w_0|}\right) \\
  &\lesssim \frac{M(\log(q_k)+1)}{r^{j^D_k}\log|w_0|} 
   \lesssim \frac{Mr^5\log q_k}{\log q_{k+1}} < \frac{Mr^5}{N} = \frac{M}{4M'}.  
\end{align*}
Finally, we turn to the possibility that $j_i^B\leq j< j_i^C$ ($i\ge1$).  
Since $j^C_k-j^B_k\le \log_r(3M')+4$, we have 
\begin{align*}
  \frac{-\log(\rho_j\delta_j)}{r^j(r-1)\log|w_0|} 
  &\le M\left(\frac{\log q_{\Tilde{k}}-\log(2c_0\rho^{N'})}{r^{j^B_i}\log|w_0|}\right) 
  \lesssim \frac{MM'r^4\log q_{\Tilde{k}}}{r^{j^C_k}\log|w_0|} \\
  &\le \frac{Mr^4\log q_{\Tilde{k}}}{\log q_k} < \frac{Mr^4}{N} < \frac{M}{4M'} 
\end{align*}
  and the lemma follows.  
\end{proof}

The proof of Theorem~\ref{thm:dichotomy} in the case where $l_N$ is infinite will be complete with the proof of the following lemma. 
\begin{lemma}\label{lem:cross}
If $(\lambda,\mu)$ satisfies (\ref{PM:conv}) then $F\subset\NE(X_{\lambda,\mu},\omega)$.  
\end{lemma}
\begin{proof}
It suffices to check that $\sum\delta_j<\infty$ for in that case, 
every sequence $\dots, w_j, v_j, w_{j+1},\dots$ constructed above satisfies (\ref{ieq:sumx}) and $F\subset\NE(X_{\lambda,\mu})$, by Theorem~\ref{thm:sumx}.  
We break the sum into three intervals: $j^B_k\le j\le j^C_k$, 
  $j^C_k<j<j_k^D$, and $j_k^D\leq j<j^B_{k'}$.  

Let $2n_k=\log_{q_k}q_{k+1}$ be such that $q_{k+1}^{1/2}=q_k^{n_k}$.  
It follows easily from the definitions that 
  $$j^D_k-j^C_k < \log_r n_k < \frac{\log\log q_{k+1}}{\log r}$$ 
  so that (\ref{PM:conv}) implies 
  $$\sum_{k\in\ell_N}\sum_{j^C_k<j<j^D_k}|w_j\times v_j| \le 
    \frac{4}{\log r}\sum_{k\in\ell_N}\frac{\log\log q_{k+1}}{q_k}<\infty.$$
Since $j_k^C-j_k^B\le\log_r(3M')+4$ we have 
  $$\sum_{k\in\ell_N}\sum_{j^B_k\le j\le j^C_k}|w_j\times v_j| 
     \le \sum_{k\in\ell_N} \frac{8\rho^{N'}(\log_r(3M')+4)}{q_{\Tilde{k}}}<\infty.$$  
Finally, 
  $$\sum_{k\in\ell_N}\sum_{j^D_k\le j< j^B_{k'}}|w_j\times v_j| 
      \le \sum_{k\in \ell_N}\frac{2R\rho^{N'}}{q_k}<\infty$$  
  where $R=\sum_{j\ge0}r^{-j}$.  
\end{proof}

\section{A special case where $\ell_N$ is finite}\label{s:A specail case}
We recall the definition of $\ell_N$ given in (\ref{def:ell_N}). If $\ell_N$ is finite, then we can choose a large number $k_0$ such that 
\begin{align}\label{def:k_0 for ell_N is finite}
    q_{k_0}>\max(480\rho^{N'+3}/c_0,c^M,48C')
\end{align}
and $q_k\notin \ell_N$ for all $k\geq k_0$. 

We begin by choosing a slit $w_0$ with $|w_0|>q_{k_0}^{M'}$ such that $w_0$ is not miracle and $w_0$ has a convergent $q$ with $|w_0|<q<|w_0|^r$ 

\begin{lemma}\label{lem:specail normal case}
There exists some slit $w_0$ with $|w_0|>q_{k_0}^{M'}$ such that $w_0$ has a convergent $q$ with $|w_0|<q<|w_0|^r$.
\end{lemma}
\begin{proof}
Let $w$ with $|w|>q_{k_0}^{M'}$ and $q$ be the first convergent of $w$ such that $q>3|w|$. We denote $\beta=\frac{q+1}{|w|}$ and if $\beta<|w|^{r-1}$ we finish the proof. Otherwise, $\beta\geq |w|^{r-1}$. By definition, $w$ is $(3,\beta)$-good. Then by applying Lemma~\ref{lem:good:children}, $$\#\Delta(w,3,\beta)\geq \frac{c_0\beta}{3}\geq\frac{c_0|w|^{r-1}}{3}>180.$$ 
Let $2v=2(p,q)\in \Delta(w,3,\beta)-w$, and $w'=w+kv=(\lambda+m',\mu+n')$ be such that $|w'|^{r-1}>\beta$ for some integer $k\geq2$. Then
we have 
$$\left|\frac{\lambda+m'}{\mu+n'}-\frac{p}{q}\right| = \frac{|w'\times v|}{|w'||v|}<\frac{|w\times v|}{k|v|^2}<\frac{1}{2q^2}.$$ 
Then by (\ref{convergents}), $v$ is a convergent of $w'$.  
Let $q'$ be the height of the next convergent of $w'$.  
Then by (\ref{align:estimate for convergents}), we have
\begin{equation}\label{ieq:cfc'}
  \frac{1}{q'+q} < \frac{|w\times v|}{|w'|} < \frac{1}{q'}.
\end{equation}
From the left hand side above we have 
$$q' > \frac{|w'|}{|w\times v|} - q > (3-\frac{1}{k})|w'|>2|w'|.$$  
for $|w'|>k|v|=kq$, and from the right hand side of (\ref{ieq:cfc'})
$$q' < \frac{|w'|}{|w\times v|} < \beta|w'|<|w'|^r$$  
for $|w'|^{r-1}>\beta$.
\end{proof}

By Lemma~\ref{Lem:danger slit}, we can always choose a slit $w$ that satisfies Lemma~\ref{lem:specail normal case} is not a miracle slit.
The next lemma shows that this slit is $\alpha$-normal where $\alpha=\frac{2}{\rho^{N'}}$ \footnote{To simplify the exposition, we choose to omit the case when $(\lambda, \mu)$ is uniquely rational, as the proof is quite similar to that of Lemma~\ref{lem:Uniquly N'-good=>normal} and Lemma~\ref{lem: specail good=>normal}.}.

\begin{lemma}\label{lem: specail good=>normal}
Let $w$ with $|w|>q_{k_0}^{M'}$ is not a miracle slit. If $w$ is $(2,|w|^{r-1})$-good, then it is $\alpha$-normal.
\end{lemma}
\begin{proof}
We use the similar argument with the proof in Lemma~\ref{lem:N'-good=>normal} and prove by contradiction. Assume $w$ is not $\alpha$-normal.  Let $v^\prime=(p^\prime,q^\prime)$ be the next convergent. If $q^\prime\leq |w|^{1+(r-1)N^\prime}$, then $w$ is $\alpha$-normal by definition. So we must have
    \begin{align*}
        q^\prime>|w|^{1+(r-1)N^\prime}.
    \end{align*}
    Then 
    \begin{align*}
        \frac{q^\prime}{|w|}>|w|^{(r-1)N^\prime}\geq q^{(1-\frac{1}{r})N^\prime}=q^{4(N+1)}.
    \end{align*}
    So it can be deduced that
    \begin{align*}
        |w\times v|<\frac{|w|}{|v^\prime|}<\frac{1}{q^{4(N+1)}}.
    \end{align*}
    Writing $w=\left(\lambda+m,\mu+n\right)$ we have
    \begin{align*}
        |q(\lambda+m)-p(\mu+n)|<\frac{1}{q^{4(N+1)}}.
    \end{align*}
    Equivalently,
    \begin{align}\label{align:eta_l'}
        |q\lambda-p\mu+(qm-pn)|<\frac{1}{q^{4(N+1)}}.
    \end{align}
Let $(\ell,\eta_j)$ be the pair of lines associated to $w$ for $l=(q,-p,qm-pn)$ by definition.  Since $w$ is not a miracle slit, by hypothesis, we have only two cases to consider.

Case 1: $\ell\neq\pm\eta_j$\\
 By the definition of nearest rational affine line, we have
    \begin{align}\label{aligh:eta_j'}
        |a_j\lambda+b_j\mu+c_j|<\frac{1}{q^{4(N+1)}}.
    \end{align}
Note that $\ell$ and $\eta_j$ are not parallel. If not, we notice the distance between $\ell$ and $\eta$ is larger than $\frac{1}{2q}$. While the distances from $(\lambda,\mu)$ to $\ell$ and $\eta_j$ are both less than $\frac{1}{q^{4(N+1)}}$, which is a contradiction. 

The intersection point of $\ell$ and $\eta_j$ is 
\begin{align*}
        \left(\frac{c_jp+(qm-pn)b_j}{a_jp+b_jq},\frac{c_jq+(pn-qm)a_j}{a_jp+b_jq}\right) =: \left(\frac{p_{h,1}}{q_h},\frac{p_{h,2}}{q_h}\right) =: \dot{v}_h
\end{align*}
in lowest terms. Recalling that $\eta_j^-\le\ell^-$ 
we have    $$q_h\leq |a_jp+b_jq| \leq 2\eta_j^-q\leq 4q^2.$$
From (\ref{align:eta_l'}) and (\ref{aligh:eta_j'}) it follows that 
\begin{align*}
      \|(\lambda,\mu)-\dot{v}_h\|_{\infty}<\frac{2\max(|a_j+q|,|b_j+p|)}{q^{4(N+1)}|a_jp+b_jq|}<\frac{6}{q^{4N+3}}.
    \end{align*}
    By the definition of $C^\prime$,
    \begin{align*}
         \|(\lambda,\mu)-\dot{v}_h\| < C^\prime\|(\lambda,\mu)-\dot{v}_h\|_{\infty}<\frac{1}{2(4q^2)^2}<\frac{1}{2q_h^2}
    \end{align*}
     which implies that $\dot{v}_h$ is a best approximation vector.
     By (\ref{C_2}), we have
     \begin{align}
         \frac{1}{2q_hq_{h+1}}<\frac{6C'}{q^{4N+3}}.
     \end{align}
     Since $q>|w|>48C'$, we have
     \begin{align*}
         q_{h+1}>\frac{q^{4N+3}}{12C'q_h}>\frac{q^{4N+1}}{48C'}>q^{3N}>(4q^2)^N\geq(q_h)^N
     \end{align*}
     from which it follows that $h\in l_N$.
 So by definition of $q_{k_0}$, we have $q_{h+1}\leq q_{k_0}$.
 While $q_{h+1}>q^{3N}>q_{k_0}^{3N}$ which is a contradiction.

Case 2: $\ell=\pm\eta_j$ and 
    \begin{align}\label{|eta_j+1|'}
        \eta_{j+1}^-<(\eta_j^-)^{r^3}<(2q)^{r^3}.
    \end{align}
    By the definition of nearest oriented affine hyperplane, we have
    \begin{align}\label{align:eta_j+1'}
        |a_{j+1}\lambda+b_{j+1}\mu+c_{j+1}|<\frac{1}{q^{4(N+1)}}
    \end{align}
   
    The intersection of $\eta_j$ and $\eta_{j+1}$ is $$\left(\frac{c_{j+1}p+(qm-pn)b_{j+1}}{a_{j+1}p+b_{j+1}q},\frac{c_{j+1}q+(pn-qm)a_{j+1}}{a_{j+1}p+b_{j+1}q}\right)=: \left(\frac{p_{h,1}}{q_h},\frac{p_{h,2}}{q_h}\right) =: \dot{v}_h$$ 
    in lowest terms. So it follows by (\ref{|eta_j+1|'}) and $r^3<3/2$ that $q_h\leq a_{j+1}p+b_{j+1}q\leq 2|\eta_{j+1}|q\leq 2^{r^3+1}q^{r^3+1} <q^3$. Then from (\ref{align:eta_l'}) and (\ref{align:eta_j+1'}) it follows that 
    \begin{align*}        
        \|(\lambda,\mu)-\dot{v}_h\|_{\infty}<\frac{2\max(|a_{j+1}+q|,|b_{j+1}+p|)}{q^{4(N+1)}|a_{j+1}p+b_{j+1}q|}<\frac{1}{q^{4N+1}}.
    \end{align*}
    With the same argument in case 1, we have
    \begin{align*}
        \|(\lambda,\mu)-\dot{v}_h\| < C^\prime\|(\lambda,\mu)-\dot{v}_h\|_{\infty}<\frac{1}{2(4q^2)^2}<\frac{1}{2q_h^2}
    \end{align*}
    which implies that $\dot{v}_h$ is a best approximation vector. 
    Since
     \begin{align}
         \frac{1}{2q_hq_{h+1}}<\frac{C'}{q^{4N+1}}.
     \end{align}
     we have
    \begin{align*}
         q_{h+1}>\frac{q^{4N+1}}{2C'q_h}>(q^3)^N>(q_h)^N
     \end{align*}
    which implies $h\in l_N$. We follow the previous prove in case 1 and get $q_{h+1}\leq q_{k_0}$
    which is also impossible.  
\end{proof}

Then one can check that the conclusion in Proposition~\ref{prop:normal} will holds for slits $w$ when replacing the first norm condition by $|w|>q_{k_0}^{M'}$.

We set $ \alpha_j=\alpha r^j, \delta_j = \frac{1}{\alpha r^j}=\frac{\rho^{N'}}{2r^j}$ and $\rho_j=\frac{c_0}{4\rho^{N'+1}}$. 
\begin{lemma}
For $j\geq 1$, every slit $w$ of level $j$ is $\alpha_kr^{j}$-normal.  Moreover, the cross-products of each slit of level $j$ with its children are less than $\delta_j$ and the number of children is at least $\rho_j|w|^{r-1}\delta_j$.  
\end{lemma}
\begin{proof}
The case $j=0$ of the first assertion follows from Lemma~\ref{lem: specail good=>normal} while the remaining cases follow from Lemma~\ref{lem: specail good=>normal} and Proposition~\ref{prop:normal}.  

When we apply Proposition~\ref{prop:normal}, we should firstly verify the inequality (\ref{ieq:normal}) holds, i.e. if 
\begin{equation}\label{ieq:normal:2'}
  60q_{k_0}^2 r^{2j}\rho^{N'+3}\le c_0|w|^{(r-1)^2} 
\end{equation}
and
\begin{equation}
  |w|^{(r-1)^2}>c.
\end{equation}
The second condition is given by (\ref{ieq:w_0>5}).
We also note that it is enough to check (\ref{ieq:normal:2'}) in the case $j=0$ since the left hand side increases by a factor $r^2$ as $j$ increments by one, while the right hand side increases by a factor $|w|^{(r-1)^3}>q_{k_0}^{3(r-1)}>5>r^2$.  
Moreover, since $|w|^{(r-1)^2}>q_k^3$, (\ref{ieq:normal:2'}) in the case $j=0$ follows from $480\rho^{N'+3}<c_0q_{k_0}$, which is guaranteed by the third term in (\ref{def:k_0}).  
  
Then we apply Proposition~\ref{prop:normal} to an $\alpha$-normal slit, the cross-products with its children are less than $1/\alpha$, which is $\delta_j$ if $\alpha=\alpha_kr^{j}$.  
The number of children is at least $$\frac{c_0|w|^{r-1}}{4\alpha\rho^{N'+1}} 
   = \rho_j|w|^{r-1}\delta_j.$$ 
\end{proof}

Next, we need to first demonstrate that the tree structure constructed under this special case satisfies the conditions required in Lemma~\ref{lem:gaps}. By verifying condition (\ref{ieq:sumx}), we can then prove that a subset of non-ergodic directions can be constructed. Finally, using the Falconer's lower bound estimate  (equation (\ref{eq:Falc})) in \S\ref{s:Lower}, we conclude that its Hausdorff dimension is \( 1/2 \).

\begin{lemma}\label{lem:delta_j,m_j}
There exists $J \geq 0$ such that for all $j \geq J$, $\delta_j < \tfrac{1}{16}$ and $m_j \geq 2$.
\end{lemma}
\begin{proof}
By definition, $\delta_j = \frac{\rho^{N'}}{2r^j}$. There exists $J$ such that $\delta_J \leq \tfrac{1}{16}$.  
Moreover, since $m_j = \rho_j|w_0|^{r^j(r-1)}\delta_j$, and the growth rate of $|w_0|^{r^j(r-1)}$ exceeds the decay rate of $\rho_j\delta_j$, it suffices to verify the case $j = 1$.  
For $j = 1$, using the fact that $|w|^{r-1} > q_{k_0} > \frac{16\rho}{c_0}$, we have  
$$
m_1 \geq \frac{c_0|w|^{r-1}}{8\rho} \geq 2.
$$  
\end{proof}

We now prove that the constructed set $F \subset \NE(X_{\lambda,\mu},\omega)$.

\begin{lemma}
For the $J$ constructed in Lemma~\ref{lem:delta_j,m_j}, when $j \geq J$, $\sum \delta_j < \infty$.
\end{lemma}

\begin{proof}
$$
\sum_{j\geq J}|w_j\times v_j| 
\leq \sum_{j\geq J}\frac{\rho^{N'}}{2r^j}<\infty.
$$
\end{proof}

Finally, we apply the Falconer-type lower bound on the dimension from \S\ref{ss:Cantor} to the estimate (\ref{eq:Falc}), and deduce that the Hausdorff dimension is equal to \( \frac{1}{2} \).

\begin{lemma}
\[
\liminf_{j \to \infty} d_j > \frac{1}{2} - \varepsilon.
\]
\end{lemma}

\begin{proof}
By \eqref{eq:local}, it suffices to show that both terms in \eqref{eq:num} and \eqref{eq:den} are controlled by \( \delta \).  
Using the choice of \( M \) and \( M' \) in \eqref{def:M'}, for sufficiently large \( j \), it is enough to ensure that each term is small.

We first consider the expression in \eqref{eq:den}. By \eqref{def:k_0 for ell_N is finite}, the first term satisfies
\begin{align}
\frac{2r\log 5}{(r-1)\log|w_0|} \leq \frac{2r}{M}.
\end{align}
Moreover, by the definition of \( \rho_j \) and \( \delta_j \), we have
\begin{align}
\frac{\log(\rho_j\delta_j/\rho_{j+1}\delta_{j+1})}{r^j(r-1)\log|w_0|} \leq \frac{\log r}{(r-1)\log|w_0|} \lesssim 0.
\end{align}

Now we consider the expression in \eqref{eq:num}. For sufficiently large \( j \), since \( jr^{-j} \log r \leq 1 \), we obtain
\begin{align}
\frac{-\log(\rho_j\delta_j)}{r^j(r-1)\log|w_0|} 
\leq M \left( \frac{\log 8 + j \log r + \log(\rho/c_0)}{r^j \log|w_0|} \right) 
\lesssim \frac{M}{M'}.
\end{align}
This completes the proof.
\end{proof}

\bibliography{references}

\begin{thebibliography}{CHM11}

\bibitem[CC24]{Ch-Che}
Y.~Cheung and N.~Chevallier.
\newblock L\'evy-khintchin theorem for best simultaneous {Diophantine} approximations.
\newblock {\em Annales Scientifiques de l'\'Ecole Normale Sup\'erieure}, 57(1):185--240, 2024.

\bibitem[CE07]{CE}
Y.~Cheung and A.~Eskin.
\newblock Slow divergence and unique ergodicity.
\newblock 2007.

\bibitem[Che03]{Ch1}
Y.~Cheung.
\newblock Hausdorff dimension of the set of nonergodic directions (with an appendix by {M}. {B}oshernitzan).
\newblock {\em Ann. of Math. (2)}, 158:167--189, 2003.

\bibitem[Che04]{Ch2}
Y.~Cheung.
\newblock Slowly divergent geodesics in moduli space.
\newblock {\em Conform. Geom. Dyn.}, 8:167--189, 2004.

\bibitem[Che11]{Ch3}
Y.~Cheung.
\newblock Hausdorff dimension of the set of singular pairs.
\newblock {\em Ann. of Math. (2)}, 173(1):127--167, 2011.

\bibitem[CHM11]{2011Dichotomy}
Y.~Cheung, P.~Hubert, and H.~Masur.
\newblock Dichotomy for the {H}ausdorff dimension of the set of nonergodic directions.
\newblock {\em Invent. Math.}, 183:337--383, 2011.

\bibitem[Hua25]{Huang}
Y.~Huang.
\newblock On the {H}ausdorff dimension of the set of nonergodic directions.
\newblock {\em Isr. J. Math.}, 2025.

\bibitem[Khi64]{Kh}
A.~Ya. Khinchin.
\newblock {\em Continued {F}ractions}.
\newblock University of Chicago Press, 1964.

\bibitem[KS86]{KMS}
H.~Kerckhoff, S.~Masur and J.~Smillie.
\newblock Ergodicity of billiard flows and quadratic differentials.
\newblock {\em Ann. of Math.}, v. 126(no. 2):293--311, 1986.

\bibitem[Lag82]{lag2}
J.~C. Lagarias.
\newblock Best simultaneous {Diophantine} approximations. {I}. growth rates of best approximation denominators.
\newblock {\em Transactions of the American Mathematical Society}, 272(2):545--554, 1982.
\newblock {Diophantine} and {I} in the title are wrapped in braces to preserve capitalization and numbering.

\bibitem[Lag85]{lag1}
J.~C. Lagarias.
\newblock The computational complexity of simultaneous {Diophantine} approximation problems.
\newblock {\em SIAM Journal on Computing}, 14(1):196--209, 1985.
\newblock {Diophantine} in the title is wrapped in braces to preserve capitalization in BibTeX.

\bibitem[Mas92]{Ma2}
H.~Masur.
\newblock Hausdorff dimension of the set of nonergodic foliations of a quadratic differential.
\newblock {\em Duke Math. J.}, 66(no. 3):387--442, 1992.

\bibitem[Mos07]{Mos}
N.~Moshchevitin.
\newblock Best {Diophantine} approximations: the phenomenon of degenerate dimension.
\newblock In {\em Surveys in geometry and number theory: reports on contemporary Russian mathematics}, volume 338 of {\em London Mathematical Society Lecture Note Series}, pages 158--182. Cambridge University Press, Cambridge, 2007.
\newblock {Diophantine} in the title is wrapped in braces to preserve capitalization.

\bibitem[MS91]{MS}
H.~Masur and J.~Smillie.
\newblock Hausdorff dimension of sets of nonergodic folitions.
\newblock {\em Ann. of Math.}, 134:455--543, 1991.

\bibitem[PM93]{PM}
R.~P\'erez~Marco.
\newblock Sur les dynamiques holomorphes non lin\'earisables et une conjecture de v. i. arnol'd ({F}rench. {E}nglish summary) [{N}onlinearizable holomorphic dynamics and a conjecture of {V}. {I}. {A}rnol'd].
\newblock {\em Ann. Sci. Ecole Norm. Sup.}, 26(no. 5):565--644, 1993.

\bibitem[Rog51]{Rog}
C.~A. Rogers.
\newblock The signature of the errors of some simultaneous {D}iophantine approximations.
\newblock {\em Proceedings of the London Mathematical Society}, 52:186--190, 1951.

\bibitem[Vee89]{Ve2}
W.~Veech.
\newblock Teichmüller curves in moduli space, {E}isenstein series and an application to triangular billiards.
\newblock {\em Invent. Math.}, 97(3):553--583, 1989.

\end{thebibliography}
\bibliographystyle{alpha}

\end{document}